\documentclass[a4paper,10pt]{amsart}

\usepackage{amsmath}
\usepackage{amsthm}
\usepackage{amsfonts}
\usepackage{amssymb}
\usepackage{euscript}
\usepackage[all]{xy}

\newcommand{\nc}{\newcommand}

\numberwithin{equation}{section}
\newtheorem{thm}{Theorem}[section]
\newtheorem{prop}[thm]{Proposition}
\newtheorem{lem}[thm]{Lemma}
\newtheorem{cor}[thm]{Corollary}
\theoremstyle{remark}
\newtheorem{rem}[thm]{Remark}
\newtheorem{definition}[thm]{Definition}
\newtheorem{example}[thm]{Example}

\nc{\gl}{\mathfrak{gl}}
\nc{\GL}{\mathfrak{GL}}
\nc{\g}{\mathfrak{g}}
\nc{\gh}{\widehat\g}
\nc{\h}{\mathfrak{h}}
\nc{\la}{\lambda}
\nc{\al}{\alpha }
\nc{\ve}{\varepsilon }
\nc{\om}{\omega }

\nc{\ta}{\theta}
\nc{\veps}{\varepsilon}
\nc{\ch}{{\mathop {\rm ch}}}
\nc{\Tr}{{\mathop {\rm Tr}\,}}
\nc{\Id}{{\mathop {\rm Id}}}
\nc{\ad}{{\mathop {\rm ad}}}
\nc{\bra}{\langle}
\nc{\ket}{\rangle}
\nc{\x}{{\bf x}}
\nc{\bs}{{\bf s}}
\nc{\bp}{{\bf p}}
\nc{\bc}{{\bf c}}
\nc{\pa}{\partial}
\nc{\ld}{\ldots}
\nc{\cd}{\cdots}
\nc{\hk}{\hookrightarrow}
\nc{\T}{\otimes}
\newcommand{\bea}{\begin{equation}}
\newcommand{\ena}{\end{equation}}
\nc{\gr}{\mathrm{gr}}
\nc{\ov}{\overline}

\nc{\cO}{\mathcal O}
\nc{\cF}{\mathcal F}
\nc{\cL}{\mathcal L}
\nc{\msl}{\mathfrak{sl}}
\nc{\mgl}{\mathfrak{gl}}
\nc{\U}{\mathrm U}
\nc{\V}{\EuScript V}
\nc{\bH}{\EuScript H}
\nc{\Res}{\mathrm{Res\ }}

\newcommand{\dimv}{{\bf dim}}

\begin{document}
\title{Desingularization of quiver Grassmannians for Dynkin quivers}
\author{G. Cerulli Irelli, E. Feigin, M. Reineke}
\address{Giovanni Cerulli Irelli:\newline
Mathematisches Institut, Universit\"at Bonn, Endenicher Allee 60, D - 53115 Bonn, Germany}
\email{cerulli.math@googlemail.com}
\address{Evgeny Feigin:\newline
Department of Mathematics, National Research University Higher School of Economics, Vavnilova str. 7, 117312 Moscow, Russia\newline
{\it and }\newline
Tamm Department of Theoretical Physics,
Lebedev Physics Institute, Leninskij prospekt 53, 119991 Moscow, Russia
}
\email{evgfeig@gmail.com}
\address{Markus Reineke:\newline
Fachbereich C - Mathematik, Bergische Universit\"at Wuppertal, D - 42097 Wuppertal, Germany}
\email{reineke@math.uni-wuppertal.de}
\begin{abstract} A desingularization of arbitrary quiver Grassmannians for representations of Dynkin quivers is constructed in terms of quiver Grassmannians for an algebra derived equivalent to the Auslander algebra of the quiver.
\end{abstract}
\maketitle

\section{Introduction}

Quiver Grassmannians are varieties parametrizing subrepresentations of a quiver representation. They first appeared in \cite{CrawleyTree,SchofieldGeneric} in relation to questions on generic properties of quiver representations; later it was observed (see \cite{CC}) that these varieties play an important role in cluster algebra theory \cite{FZI} since cluster variables can be described in terms of the Euler characteristic of quiver Grassmannians. A rather well-studied specific class of quiver Grassmannians are the varieties of subrepresentations of exceptional quiver representations since they are smooth projective varieties, see for example \cite{CR}.\\[1ex]
In \cite{CFR,CFR2} the authors of the present paper initiated a systematic study of a class of singular quiver Grassmannians, starting from the observation that the type $A$ degenerate flag varieties of \cite{F2010,F2011,FF} are quiver Grassmannians. Namely, Grassmannians of subrepresentations of the direct sum $P\oplus I$ of a projective representation $P$ and an injective representation $I$ of a Dynkin quiver $\mathcal{Q}$, of the same dimension vector as $P$, are considered. They are shown in \cite{CFR} to be reduced irreducible normal local complete intersection varieties, admitting a group action with finitely many orbits, as well as a cell decomposition. Moreover, a detailed description of the singular locus of degenerate flag varieties is given in \cite{CFR2}.\\[1ex]
In pursuing the analysis of singular quiver Grassmannians, it is thus desirable to have an explicit desingularization at our disposal;
this is done in \cite{FF} for the type $A$ degenerate flag varieties.\\[1ex]
The main result (see Section \ref{section7}) of the present work is that an appropriate representation theoretic re-interpretation of the construction of \cite{FF} generalizes, and provides desingularizations of (irreducible components of) arbitrary quiver Grassmannians over Dynkin quivers. In fact, the desingularization is itself (an irreducible component of) a quiver Grassmannian over a certain quiver $\widehat{\mathcal{Q}}$, and can be described explicitly once the irreducible components of the quiver Grassmannian to be desingularized are known. Moreover, every fibre of the desingularization map is described in terms of a quiver Grassmannian over $\widehat{\mathcal{Q}}$ itself.\\[1ex]
The authors are confident that the desingularization is useful for a more refined analysis of the geometry of certain classes of singular quiver Grassmannians, with respect to, for example, Frobenius splitting, cohomology of line bundles, or intersection cohomology, in the spirit of \cite{FF}.\\[1ex]
At the heart of the construction of the desingularization lies the definition of a certain algebra $B_\mathcal{Q}$ of global dimension at most two (see Section \ref{section4}), together with a fully faithful functor $\Lambda$ from the category of representations of $\mathcal{Q}$ to the one of $B_\mathcal{Q}$ with special homological properties.
Namely, the essential image of $\Lambda$ consists of representations of $B_\mathcal{Q}$ without self-extensions and of projective and injective dimension at most $1$ (see Section \ref{section5}).\\[1ex]
The algebra $B_\mathcal{Q}$ arises as the endomorphism ring of an additive generator of a certain category
$\mathcal{H}_\mathcal{Q}$ of embeddings of projective representations of $\mathcal{Q}$ (see Section \ref{section3}); it is derived equivalent (but not Morita equivalent) to the Auslander algebra of $\mathcal{Q}$ (see Section \ref{section6}).\\[1ex]
The algebra $B_\mathcal{Q}$ and the functor $\Lambda$ should be of independent interest. The authors believe that the associated representation varieties are related to graded Nakajima quiver varieties in the spirit of \cite{HL,LP}, and that they admit applications to the geometry of orbit closures in representation varieties of $\mathcal{Q}$. These topics will be discussed elsewhere.\\[1ex]
Although (or precisely because?) the construction of the desingularization is a very general and conceptual one, it is still nontrivial to analyse, say, the dimension of the singular fibres; no general formulas or estimates are known at the moment. But a detailed analysis of all cases in type $A_2$ is given, suggesting that good geometric properties of the desingularization (like, for example, being one to one precisely over the smooth locus) can only be expected precisely in the case considered above, namely the quiver Grassmannians generalizing the type $A$ degenerate flag varieties (see Section \ref{section8}).\\[1ex]
The paper is organized as follows. In Section~\ref{section2} we collect some standard material on
categories of functors and Auslander-Reiten theory. In Section~\ref{section3} the category
$\mathcal{H}_\mathcal{Q}$ is introduced and in Section~\ref{section4} the homological properties of the
category $\bmod\,\mathcal{H}_\mathcal{Q}^{\rm op}$ are studied. In Section~\ref{section5} we define the key functor
$\Lambda:\bmod\, k\mathcal{Q}\rightarrow\bmod\,\mathcal{H}_\mathcal{Q}^{\rm op}$. Section~\ref{section6}
contains the computation of the ordinary quiver of the algebra $B_\mathcal{Q}$ and several examples.
In Section~\ref{section7} we use the results of the previous sections to construct the desingularizations
of quiver Grassmannians. Finally, in Section~\ref{section8} we close with examples
of desingularizations.

\section{Reminder on categories of functors and Auslander-Reiten theory}\label{section2}

We collect some standard material (see e.g. \cite[IV.6., A.2.]{ASS}, \cite[3.]{GR}) on the functorial approach to the representation theory of algebras and to Auslander-Reiten theory in particular. Throughout the paper, we will make free use of the basic concepts of Auslander-Reiten theory \cite[IV.]{ASS}, like almost split maps, Auslander-Reiten sequences, the Auslander-Reiten translation and its relation to the Nakayama functor, and the structure of the Auslander-Reiten quiver of a Dynkin quiver.\\[1ex]
Let $k$ be a field. All categories in the following are assumed to be $k$-linear, finite length (that is, all objects admit finite composition series) Krull-Schmidt (that is, the Krull-Schmidt theorem holds) categories with finite dimensional morphism spaces.\\[1ex]
For a category $\mathcal{C}$, we denote by $\mathcal{C}^{\rm op}$ its opposite category, and by $\bmod\, \mathcal{C}^{\rm op}$ the category of $k$-linear additive, contravariant functors from $\mathcal{C}$ to $\bmod\, k$, the category of finite dimensional $k$-vector spaces. For an object $M$ of $\mathcal{C}$, the functor ${\rm Hom}(\_,M)$ is an object of $\bmod\, \mathcal{C}^{\rm op}$. By Yoneda's lemma, we have $${\rm Hom}_{\bmod\,\mathcal{C}^{\rm op}}({\rm Hom}(\_,M),F)\simeq F(M)$$ for every $M\in\mathcal{C}$, $F\in\bmod\,\mathcal{C}^{\rm op}$. From this we can conclude that ${\rm Hom}(\_,M)$ is a projective object of $\bmod\,\mathcal{C}^{\rm op}$; in fact, every projective object is of this form. Dually, the injective objects in $\bmod\,\mathcal{C}^{\rm op}$ are the functors ${\rm Hom}(M,\_)^*$ for $M\in\mathcal{C}$, where $V^*$ denotes the linear dual of a  $k$-vector space $V$. Moreover, the simple objects in $\bmod\,\mathcal{C}^{\rm op}$ are parametrized by the isomorphism classes of indecomposable objects in $\mathcal{C}$: for such an object $U$, there exists a unique simple functor $S_U$ which is a quotient of ${\rm Hom}(\_,U)$ and embeds into ${\rm Hom}(U,\_)^*$.\\[1ex]
For a finite dimensional $k$-algebra $A$, let $\bmod\,A$ be the category of finite dimensional
(left) $A$-modules. We can consider the subcategory ${\rm proj}\, A$ of $\bmod\, A$, which is the category
of finite dimensional projective $A$-modules. Then $\bmod\,({\rm proj}\, A)^{\rm op}$ is equivalent
to $\bmod\, A$ by associating to a module $M$ the functor ${\rm Hom}(\_,M)$.\\[1ex]
Somewhat conversely, assume that $\mathcal{C}$ admits only finitely many indecomposables; let $U_1,\ldots,U_N$ be a system of representatives. Then $\bmod\,\mathcal{C}^{\rm op}$ is equivalent to $\bmod\,B(\mathcal{C})$, where $B(\mathcal{C}):={\rm End}_{\mathcal C}(\bigoplus_iU_i)^{\rm op}$; namely, by the above, ${\rm Hom}(\_,\bigoplus_iU_i)$ is a projective generator of $\bmod\,\mathcal{C}^{\rm op}$. In particular, we have $B({\rm proj}\, A)\simeq A$. If $A$ admits only finitely many indecomposables, the algebra $B(\bmod\, A)$ is called the Auslander algebra of $A$.\\[1ex]
The structure of $\bmod(\bmod A)^{\rm op}$ is related to Auslander-Reiten theory: the simple functors are precisely those of the form $S_U$ for an indecomposable $U$ in $\bmod\, A$, where $S_U$ (being additive) is determined on indecomposables by $S_U(U)=k$ and $S_U(V)=0$ for all $V\not\simeq U$. If $0\rightarrow{\tau U}\rightarrow B\rightarrow U\rightarrow 0$ is the Auslander-Reiten sequence ending in $U$, then
$$0\rightarrow{\rm Hom}(\_,\tau U)\rightarrow{\rm Hom}(\_,B)\rightarrow{\rm Hom}(\_,U)\rightarrow S_U\rightarrow 0$$
is a projective resolution of $S_U$ in $\bmod(\bmod\, A)^{\rm op}$. In particular, this category has global dimension at most two.\\[1ex]
For the category $\mathcal{C}$, we can consider the category ${\rm Hom}\,\mathcal{C}$ with objects being morphisms $f:M\rightarrow N$ in $\mathcal{C}$, and with morphisms from $f:M\rightarrow N$ to $f':M'\rightarrow N'$ being pairs $(\phi:M\rightarrow M',\psi:N\rightarrow N')$ of morphisms such that $\psi f=f'\phi$; composition is defined naturally. We also consider the full subcategories ${\rm Hom}^{\rm mono}\mathcal{C}$ (resp. ${\rm Hom}^{\rm iso}\mathcal{C}$) with objects the monomorphisms (resp. isomorphisms) between objects of $\mathcal{C}$.\\[1ex]
We sometimes denote by ${\rm ind}\,A$ the set of indecomposable (finite dimensional, left) $A$--modules.

\section{The category $\mathcal{H}_\mathcal{Q}$}\label{section3}

From now on, let $\mathcal{Q}$ be a Dynkin quiver with set of vertices $\mathcal{Q}_0$, and let $k\mathcal{Q}$ be its path algebra. We denote by $S_i$ the simple left module corresponding to a vertex $i$ of $\mathcal{Q}$, and by $P_i$ (resp.~$I_i$) its projective cover (resp.~injective hull). Then every object of ${\rm proj}\, k\mathcal{Q}$ is isomorphic to a finite direct sum of the $P_i$.\\[1ex]
We consider the category ${\rm Hom}^{\rm mono}({\rm proj}\, k\mathcal{Q})$; its objects are thus injective maps between projective representations of $\mathcal{Q}$. We note a few obvious formulas for morphisms in this category whose verification is immediate:

\begin{lem}\label{l30} For all injective maps $P\rightarrow Q$ and $Q\rightarrow R$ between projectives $P$, $Q$ and $R$, we have in ${\rm Hom}^{\rm mono}({\rm proj}\, k\mathcal{Q})$:
\begin{enumerate}
\item ${\rm Hom}(P\stackrel{{\rm id}}{\rightarrow}P,Q\rightarrow R)\simeq{\rm Hom}(P,Q)$,
\item ${\rm Hom}(P\rightarrow Q,R\stackrel{{\rm id}}{\rightarrow}R)\simeq{\rm Hom}(Q,R)$,
\item ${\rm Hom}(0\rightarrow P,Q\rightarrow R)\simeq{\rm Hom}(P,R)$,
\item Every epimorphism in ${\rm Hom}(P\rightarrow Q,0\rightarrow R)$ is split.
\end{enumerate}
\end{lem}

Taking the cokernel of an injective map between projectives induces a functor ${\rm Coker}:{\rm Hom}^{\rm mono}({\rm proj}\, k\mathcal{Q})\rightarrow\bmod\, k\mathcal{Q}$; on morphisms it is defined by mapping a pair $(\varphi,\psi)$ to the unique $f$ making the following diagram commutative:
$$\begin{array}{ccrcrcrcc}
0&\rightarrow&P&\stackrel{\iota}{\rightarrow}&Q&\rightarrow&{\rm Coker}\,\iota&\rightarrow 0\\
&&\varphi\downarrow&&\psi\downarrow&&f\downarrow&&\\
0&\rightarrow&P'&\stackrel{\iota'}{\rightarrow}&Q'&\rightarrow&{\rm Coker}\,\iota'&\rightarrow 0
\end{array}$$

\begin{lem}\label{l31} The functor ${\rm Coker}$ is full and dense. We have $${\rm Hom}((P\stackrel{\iota}{\rightarrow} Q),(P'\stackrel{\iota'}{\rightarrow} Q'))/{\rm Hom}(Q,P')\simeq{\rm Hom}_{\bmod\, k\mathcal{Q}}({\rm Coker}\,\iota,{\rm Coker}\,\iota').$$
\end{lem}

\begin{proof} Given $M\in\bmod\, k\mathcal{Q}$, there exists a projective resolution $0\rightarrow P\stackrel{\iota}{\rightarrow} Q\rightarrow M\rightarrow 0$, thus $\iota:P\rightarrow Q$ is an object of ${\rm Hom}^{\rm mono}({\rm proj}\, k\mathcal{Q})$ mapping to $M$ under ${\rm Coker}$. This proves density. Every map $f:M\rightarrow N$ between representations lifts over projective resolutions as in the above diagram, which proves that ${\rm Coker}$ is full. Again in the above diagram, the induced morphism $f$ on cokernels is $0$ if and only if $\psi$ factors over $\iota'$, which proves the claimed isomorphism.
\end{proof}

\begin{cor} The functor ${\rm Coker}$ induces an equivalence between the quotient category ${\rm Hom}^{\rm mono}({\rm proj}\, k\mathcal{Q}	)/{\rm Hom}^{\rm iso}({\rm proj}\, k\mathcal{Q})$ and $\bmod\, k\mathcal{Q}$.
\end{cor}

\begin{proof} If a morphism ${\rm Hom}((P\stackrel{\iota}{\rightarrow} Q),(P'\stackrel{\iota'}{\rightarrow} Q'))$ is of the form $(h\iota,\iota' h)$ for some $h\in{\rm Hom}(Q,P')$, it admits a factorization
$$(P\stackrel{\iota}{\rightarrow} Q)\stackrel{(\iota,{\rm id})}{\rightarrow}(Q\stackrel{{\rm id}}{\rightarrow}Q)\stackrel{(h,\iota'h)}{\rightarrow}(P'\stackrel{\iota'}{\rightarrow}Q').$$
Conversely, a factorization
$$(P\stackrel{\iota}{\rightarrow} Q)\stackrel{(\varphi\iota,\varphi)}{\rightarrow}(R\stackrel{{\rm id}}{\rightarrow}R)\stackrel{(\psi,\iota'\psi)}{\rightarrow}(P'\stackrel{\iota'}{\rightarrow}Q')$$
for $\varphi:P\rightarrow R$ and $\psi:R\rightarrow Q'$ yields a map $h=\psi\varphi$ as above. Combining this with the previous lemma, the statement follows.
\end{proof}

\begin{prop}\label{indecinhomprojkq} Up to isomorphism, the indecomposable objects of the category ${\rm Hom}^{\rm mono}{\rm proj}\, k\mathcal{Q}$ are the following:
\begin{enumerate}
\item $P_U\stackrel{\iota_U}{\rightarrow} Q_U$ for $0\rightarrow P_U\stackrel{\iota_U}{\rightarrow}Q_U\rightarrow U\rightarrow 0$ a minimal projective resolution of a non-projective indecomposable $U$ in $\bmod\, k\mathcal{Q}$,
\item $0\rightarrow P_i$ for $i\in \mathcal{Q}_0$,
\item $P_i\stackrel{{\rm id}}{\rightarrow} P_i$ for $i\in \mathcal{Q}_0$.
\end{enumerate}
\end{prop}

\begin{proof} Indecomposability of the objects in (ii) and (iii) is clear from indecomposability of $P_i$. Indecomposability of the objects in (i) follows from minimality of the resolution. Conversely, assume that $P\stackrel{\iota}{\rightarrow}Q$ is an indecomposable object, not of the form in (ii) or (iii). Then $\iota$ is not an isomorphism, thus $U=Q/P\not=0$. Since ${\rm End}(P\stackrel{\iota}{\rightarrow}Q)$ is local, so is ${\rm End}(U)$ by Lemma \ref{l31}, and thus $U$ is indecomposable. It is also non-projective, since $\iota$ is non-split. But then $P\stackrel{\iota}{\rightarrow}Q$ admits $P_U\stackrel{\iota_U}{\rightarrow}Q_U$ as a direct summand, proving that they are isomorphic.
\end{proof}


\begin{definition} Let $\mathcal{H}_\mathcal{Q}$ be the full subcategory of ${\rm Hom}^{\rm mono}({\rm proj}\, k\mathcal{Q})$ of objects without direct summands of the form $0\rightarrow P$. Let $B_\mathcal{Q}$ be the algebra $B(\mathcal{H}_\mathcal{Q})$.
\end{definition}

Equivalently, $\mathcal{H}_\mathcal{Q}$ is the full subcategory of embeddings $P\subset Q$ whose image is not contained in a proper direct summand. Moreover, $\mathcal{H}_\mathcal{Q}$ is equivalent to the quotient category of ${\rm Hom}^{\rm mono}({\rm proj}\, k\mathcal{Q})$ by morphisms factoring through an object $0\rightarrow P$ by Lemma \ref{l30}. The cokernel functor thus induces an equivalence between $\mathcal{H}_\mathcal{Q}/{\rm Hom}^{\rm iso}({\rm proj}\, k\mathcal{Q})$ and $\underline{\bmod}\, k\mathcal{Q}$, the quotient category of $\bmod\, k\mathcal{Q}$ by ${\rm proj}\, k\mathcal{Q}$.



\section{The category $\bmod\,\mathcal{H}_\mathcal{Q}^{\rm op}$ and its homological properties}\label{section4}

Now we consider the category $\bmod\,\mathcal{H}_\mathcal{Q}^{\rm op}$, whose objects thus are contravariant functors from embeddings $P\subset Q$ without direct summands $0\subset R$ to vector spaces.\\[1ex]
Note again that the projective objects of this functor category are the objects of the form ${\rm Hom}(\_,(P\subset Q))$; more precisely, the projective indecomposable objects are the ${\rm Hom}(\_,(P_U\subset Q_U))$ for non-projective indecomposables $U$ in $\bmod\, k\mathcal{Q}$, and the ${\rm Hom}(\_,(P_i=P_i))$ for $i\in \mathcal{Q}_0$. Dually, the injective objects on $\bmod\,\mathcal{H}_\mathcal{Q}^{\rm op}$ are of the form ${\rm Hom}(P\subset Q,\_)^*$. \\[1ex]
The cokernel functor ${\rm Coker}:\mathcal{H}_\mathcal{Q}\rightarrow\bmod\, k\mathcal{Q}$ induces an exact functor $\bmod\,(\bmod\, k\mathcal{Q})^{\rm op}\rightarrow\bmod\,\mathcal{H}_\mathcal{Q}^{\rm op}$. We will now use this functor to construct projective resolutions in $\bmod\,\mathcal{H}_\mathcal{Q}^{\rm op}$ using those in $\bmod\,(\bmod\, k\mathcal{Q})^{\rm op}$ coming from the Auslander-Reiten theory. We first need a lemma relating homomorphism spaces in $\mathcal{H}_\mathcal{Q}$ and in $\bmod\, k\mathcal{Q}$ via the cokernel functor.

\begin{lem} For all $(P\subset Q)$ in $\mathcal{H}_\mathcal{Q}$, we have an exact sequence of functors
$$0\rightarrow{\rm Hom}(\_,(P=P))\rightarrow{\rm Hom}(\_,(P\subset Q))\rightarrow{\rm Hom}({\rm Coker}\,\_,Q/P)\rightarrow 0.$$
\end{lem}

\begin{proof} The first map in the claimed sequence is induced by the canonical map $(P=P)\rightarrow(P\subset Q)$ in $\mathcal{H}_\mathcal{Q}$, whereas the second map is induced by applying the cokernel functor to both arguments of ${\rm Hom}$. We prove exactness of the sequence by evaluating on an arbitrary object $(R\subset S)$ of $\mathcal{H}_\mathcal{Q}$.  We have an induced commutative diagram with exact rows and columns
$$\begin{array}{ccccccccc}
&&0&&0&&0&&\\
&&\downarrow&&\downarrow&&\downarrow&&\\
0&\rightarrow&{\rm Hom}(S/R,P)&\rightarrow&{\rm Hom}(S/R,Q)&\rightarrow&{\rm Hom}(S/R,Q/P)&&\\
&&\downarrow&&\downarrow&&\downarrow&&\\
0&\rightarrow&{\rm Hom}(S,P)&\rightarrow&{\rm Hom}(S,Q)&\rightarrow&{\rm Hom}(S,Q/P)&\rightarrow&0\\
&&\downarrow&&\downarrow&&\downarrow&&\\
0&\rightarrow&{\rm Hom}(R,P)&\rightarrow&{\rm Hom}(R,Q)&\rightarrow&{\rm Hom}(R,Q/P)&\rightarrow&0
\end{array}$$
A standard diagram chase yields an exact sequence
$$0\rightarrow{\rm Hom}(S,P)\rightarrow Y\rightarrow{\rm Hom}(S/R,Q/P)\rightarrow 0,$$
where $Y$ denotes the subspace of ${\rm Hom}(R,P)\oplus{\rm Hom}(S,Q)$ of pairs mapping to the same element of ${\rm Hom}(R,Q)$. But this sequence immediately identifies with the evaluation of the claimed exact sequence at $(R\subset S)$.
\end{proof}

\begin{thm}\label{gldim2} The category $\bmod\,\mathcal{H}_\mathcal{Q}^{\rm op}$ has global dimension at most two.
\end{thm}

\begin{proof} We exhibit projective resolutions of the simple objects in $\bmod\,\mathcal{H}_\mathcal{Q}^{\rm op}$. These simple objects are the $S_{P_U\subset Q_U}$ for $U$ a non-projective indecomposable, and the $S_{P_i=P_i}$ for $i\in \mathcal{Q}_0$.\\[1ex]
Recall the projective resolution
$$0\rightarrow{\rm Hom}(\_,\tau U)\rightarrow{\rm Hom}(\_,B)\rightarrow{\rm Hom}(\_,U)\rightarrow S_U\rightarrow 0$$
 in $\bmod\,(\bmod\, A)^{\rm op}$, where $0\rightarrow{\tau U}\rightarrow B\rightarrow U\rightarrow 0$ is the Auslander-Reiten sequence ending in $U$. The above minimal projective resolutions of $U$ and $\tau U$ induce a (not necessarily minimal) projective resolution
$$0\rightarrow \underbrace{P_U\oplus P_{\tau U}}_{=P_B}\rightarrow \underbrace{Q_U\oplus Q_{\tau U}}_{=Q_B}\rightarrow B\rightarrow 0$$
of the middle term $B$. Together with the exact sequences of the previous lemma, we can thus consider the following commutative diagram (in which ${\rm Hom}(X,Y)$ is abbreviated to $(X,Y)$):
$$
\xymatrix@C=6pt{
       & 0\ar[d]                                &0\ar[d]                   &0\ar[d]                   &\\
0\ar[r]&(\_,(P_{\tau U}=P_{\tau U}))\ar[r]\ar[d]&(\_,(P_B=P_B))\ar[r]\ar[d]&(\_,(P_U=P_U))\ar[r]\ar[d]&0\ar[d]\\
0\ar[r]&(\_,(P_{\tau U}\subset Q_{\tau U}))\ar[r]\ar[d]&(\_,(P_B\subset Q_B))\ar[r]\ar[d]&(\_,(P_U\subset Q_U))\ar[r]\ar[d]&S_{P_U\subset Q_U}\ar[r]\ar[d]&0\\
0\ar[r]&({\rm Coker}\_,\tau U)\ar[r]\ar[d]&({\rm Coker}\_,B)\ar[r]\ar[d]&({\rm Coker}\_,U)\ar[r]\ar[d]&S_U({\rm Coker}\_)\ar[r]\ar[d]&0\\
&0&0&0&0&
}
$$
All columns being exact and the top and bottom row being exact, a double application of the $3\times 3$-lemma yields exactness of the middle row. This sequence provides the desired projective resolution of $S_{P_U\subset Q_U}$ as long as $\tau U$ is non-projective. In case it is, we note that the restriction of the functor ${\rm Hom}(\_,(0\subset \tau U))$ to $\mathcal{H}_\mathcal{Q}$ is zero, thus the middle row provides an even shorter projective resolution. It remains to exhibit a projective resolution of $S_{P_i=P_i}$, which is provided by
$$0\rightarrow{\rm Hom}(\_,(\bigoplus_{i\rightarrow j}P_j\subset P_i))\rightarrow{\rm Hom}(\_,(P_i=P_i))\rightarrow S_{P_i=P_i}\rightarrow 0.$$
This can be verified by evaluating on indecomposable objects, using that $\bigoplus_{i\rightarrow j}P_j\simeq{\rm rad}\, P_i$, and that the inclusion ${\rm rad} P_i\subset P_i$ is almost split.
The theorem is proved.
\end{proof}

\section{The functor $\Lambda$}\label{section5}

We have an embedding ${\rm proj}\, k\mathcal{Q}\rightarrow \mathcal{H}_\mathcal{Q}$ associating to $P$ the object $(P=P)$. This induces a restriction functor
${\rm res}:\bmod\,\mathcal{H}_\mathcal{Q}^{\rm op}\rightarrow\bmod\,({\rm proj}\, k\mathcal{Q})^{\rm op}\simeq\bmod\, k\mathcal{Q}$.\\[1ex]
Central to the following is the definition of a functor
$\Lambda:\bmod\, k\mathcal{Q}\rightarrow\bmod\,\mathcal{H}_\mathcal{Q}^{\rm op}$ (see Section \ref{section6} for concrete examples):

\begin{definition} For $M$  in $\bmod\, k\mathcal{Q}$, define an object $\widehat{M}$ in
$\bmod\,\mathcal{H}_\mathcal{Q}^{\rm op}$ as follows: $$\widehat{M}(\iota:P\rightarrow Q)={\rm Im}({\rm Hom}(\iota,M):{\rm Hom}(Q,M)\rightarrow{\rm Hom}(P,M)),$$
with the natural definition on morphisms. This defines a functor $\Lambda:\bmod\, k\mathcal{Q}\rightarrow\bmod\,\mathcal{H}_\mathcal{Q}^{\rm op}$.
\end{definition}

\begin{lem} We have ${\rm res}\, \widehat{M}\simeq M$ naturally.
\end{lem}

\begin{proof} We have $$ ({\rm res}\, \widehat{M})(P)=\widehat{M}(P=P)=
{\rm Im}({\rm Hom}(P,M)\stackrel{{\rm id}}{\rightarrow}{\rm Hom}(P,M))={\rm Hom}(P,M);$$
using the equivalence between $\bmod\, k\mathcal{Q}$ and $\bmod\,({\rm proj}\, k\mathcal{Q})^{\rm op}$, the statement follows.
\end{proof}

We note the following weak adjunction properties:
\begin{lem}\label{lemweak} For all $M$ in $\bmod\, k\mathcal{Q}$ and all $F$ in $\bmod\,\mathcal{H}_\mathcal{Q}^{\rm op}$, the natural maps
$${\rm Hom}(\widehat{M},F)\rightarrow{\rm Hom}({\rm res}\,\widehat{M},{\rm res}\,F)\simeq{\rm Hom}(M,{\rm res}\, F)$$
and
$${\rm Hom}(F,\widehat{M})\rightarrow{\rm Hom}({\rm res}\,F,{\rm res}\,\widehat{M})\simeq{\rm Hom}({\rm res}\, F,M)$$
are injective.
\end{lem}

\begin{proof} For every object $(P\subset Q)$ in $\mathcal{H}_\mathcal{Q}$, we have a natural chain of morphisms
$$(P=P)\rightarrow(P\subset Q)\rightarrow (Q=Q)$$
in $\mathcal{H}_\mathcal{Q}$. Applying the functor $\widehat{M}$, this induces a chain
$${\rm Hom}(Q,M)\rightarrow{\rm Im}({\rm Hom}(Q,M)\rightarrow{\rm Hom}(P,M))\rightarrow{\rm Hom}(P,M),$$
so that the first map is surjective, and the second map is injective. Suppose that $\varphi:\widehat{M}\rightarrow F$ maps to $0$ under ${\rm res}$. This yields a commutative diagram
$$\begin{array}{rcccl}
{\rm Hom}(Q,M)&\rightarrow&{\rm Im}({\rm Hom}(Q,M)\rightarrow{\rm Hom}(P,M))&\rightarrow&{\rm Hom}(P,M)\\
0\downarrow&&\downarrow&&\downarrow 0\\
F(Q=Q)&\rightarrow&F(P\subset Q)&\rightarrow&F(P=P).
\end{array}$$
The outer vertical maps being zero and the first map in the upper row being surjective, we see that the middle vertical map is zero. This being true for an arbitrary embedding $P\subset Q$, the map $\varphi$ is already zero. The second statement is proved dually.
\end{proof}


\begin{cor} The functor $\Lambda$ is fully faithful.
\end{cor}

\begin{proof}
For all $M$ and $N$ in $\bmod\, k\mathcal{Q}$, we have a chain of maps
$${\rm Hom}(M,N)\stackrel{\Lambda}{\rightarrow}{\rm Hom}(\widehat{M},\widehat{N})\stackrel{{\rm res}}{\rightarrow}{\rm Hom}({\rm res}\, \widehat{M},{\rm res}\, \widehat{N})\simeq{\rm Hom}(M,N)$$
whose composition is the identity, thus the first map is injective. The second map being injective by the previous lemma, the claim follows.
\end{proof}

\begin{rem} The functor $\Lambda$ is neither left nor right exact in general; from it being fully faithful we can at least conclude that injective and surjective maps are preserved.
\end{rem}

Now we come to the central result on the functor $\Lambda$:

\begin{thm}\label{ext2ext1} The following holds for all $M$ in $\bmod\, k\mathcal{Q}$:
\begin{enumerate}
\item\label{a)} Both the projective and the injective dimension of $\widehat{M}$ are at most one.
\item\label{b)} We have ${\rm Ext}^1(\widehat{M},\widehat{M})=0$.
\end{enumerate}
\end{thm}

\begin{proof}


We claim that if $0\rightarrow P\stackrel{\iota}{\rightarrow} Q\rightarrow M\rightarrow 0$ is a projective resolution of $M$, we have a projective resolution
$$0\rightarrow{\rm Hom}(\_,(P\stackrel{\iota}{\rightarrow} Q))\rightarrow {\rm Hom}(\_,(Q=Q))\rightarrow \widehat{M}\rightarrow 0$$
of $\widehat{M}$. First note that $\widehat{Q}\simeq{\rm Hom}(\_,(Q=Q))$ for $Q$ projective, thus the projection $Q\rightarrow M$ induces a projection $\widehat{Q}\rightarrow\widehat{M}$, thus it suffices to verify that ${\rm Hom}(\_,(P\subset Q))$ is indeed the kernel. Evaluating the above sequence on an embedding $R\subset S$, we get the sequence
$$
\xymatrix@C=6pt{
0\ar[r]&{\rm Hom}((R\subset S),(P\subset Q))\ar[r]&{\rm Hom}(S,Q)\ar[r]&{\rm Im}({\rm Hom}(S,M)\rightarrow{\rm Hom}(R,M))\ar[r]& 0
}
$$
whose exactness follows from the inspection of the following diagram, noting that ${\rm Hom}((R\subset S),(P\subset Q))$ equals the space of pairs in ${\rm Hom}(R,P)\oplus{\rm Hom}(S,Q)$ mapping to the same element of ${\rm Hom}(R,Q)$:
$$\begin{array}{ccccccccc}
0&\rightarrow{\rm Hom}(S,P)&\rightarrow&{\rm Hom}(S,Q)&\rightarrow&{\rm Hom}(S,M)&\rightarrow&0\\
&\downarrow&&\downarrow&&\downarrow&&\\
0&\rightarrow{\rm Hom}(R,P)&\rightarrow&{\rm Hom}(R,Q)&\rightarrow&{\rm Hom}(R,M)&\rightarrow&0.\end{array}$$
To construct an injective coresolution of $\widehat{M}$, we use the inverse Nakayama functor $\nu^-={\rm Hom}((k\mathcal{Q})^*,\_)$ which induces an equivalence between the full subcategory of $\bmod\, k\mathcal{Q}$ of injective modules and ${\rm proj}\, k\mathcal{Q}$, namely ${\rm Hom}(\_,I)\simeq{\rm Hom}(\nu^-I,\_)^*$. If $M$ itself is an injective $I$, then $\widehat{I}\simeq{\rm Hom}((\nu^-I=\nu^-I),\_)^*$ is an injective object of $\bmod\,\mathcal{H}_\mathcal{Q}^{\rm op}$, since
$$\widehat{I}(P\subset Q)={\rm Im}({\rm Hom}(Q,I)\rightarrow{\rm Hom}(P,I))={\rm Hom}(P,I)\simeq$$
$$\simeq{\rm Hom}(\nu^-I,P)^*\simeq{\rm Hom}((\nu^-I=\nu^-I),(P\subset Q))^*.$$
Otherwise we can assume $M$ to be without injective direct summands and choose an injective coresolution $0\rightarrow M\rightarrow I\rightarrow J\rightarrow 0$. By definition, we have $\nu^-M=0$, yielding an embedding $\nu^-I\subset\nu^-J$. Similar to the above case of a projective resolution, and making use of the Nakayama functor, we can verify that
$$0\rightarrow\widehat{M}\rightarrow{\rm Hom}((\nu^-I=\nu^-I),\_)^*\rightarrow{\rm Hom}((\nu^-I\subset\nu^-J),\_)^*\rightarrow 0$$
is an injective coresolution of $\widehat{M}$.\\[1ex]
To prove the second part of the theorem, we apply ${\rm Hom}(\_,\widehat{M})$ to the above projective resolution of $\widehat{M}$ and get
$$0\rightarrow{\rm Hom}(\widehat{M},\widehat{M})\rightarrow{\rm Hom}({\rm Hom}(\_,(Q=Q)),\widehat{M})\rightarrow$$
$$\rightarrow {\rm Hom}({\rm Hom}(\_,(P\rightarrow Q)),\widehat{M})\rightarrow{\rm Ext}^1(\widehat{M},\widehat{M})\rightarrow 0.$$
The first term equals ${\rm Hom}(M,M)$, and the second and third term can be computed using Yoneda's lemma, yielding the sequence
$$0\rightarrow{\rm Hom}(M,M)\rightarrow \widehat{M}(Q=Q)\rightarrow \widehat{M}(P\subset Q)\rightarrow{\rm Ext}^1(\widehat{M},\widehat{M})\rightarrow 0.$$
By the definition of $\widehat{M}$, this reads
$$0\rightarrow{\rm Hom}(M,M)\rightarrow{\rm Hom}(Q,M)\stackrel{\alpha}{\rightarrow}{\rm Im}({\rm Hom}(Q,M)\stackrel{\alpha}{\rightarrow}{\rm Hom}(P,M))\rightarrow$$
$$\rightarrow{\rm Ext}^1(\widehat{M},\widehat{M})\rightarrow 0.$$
We see that the second map is tautologically surjective, thus the desired vanishing follows.
The theorem is proved.
\end{proof}

\section{The algebra $B_\mathcal{Q}$}\label{section6}

By the results of the previous section, the utility of the algebra $B_\mathcal{Q}=B(\mathcal{H}_\mathcal{Q})$ is the following: it is an algebra of global dimension at most two, such that the original module category $\bmod\, k\mathcal{Q}$ embeds into the subcategory of $\bmod\, B_\mathcal{Q}$ of objects of projective and injective dimension at most one, in such a way that all non-trivial extensions in $\bmod\, k\mathcal{Q}$ vanish after the embedding. In contrast, the natural embedding $M\mapsto{\rm Hom}(\_,M)$ of $\bmod\, k\mathcal{Q}$ into $\bmod\,(\bmod\,k\mathcal{Q})^{\rm op}$ in general yields projective functors of injective dimension two. We will see in Proposition \ref{grsmooth} why all these properties of $\Lambda$ are essential for the construction of desingularizations of quiver Grassmannians.\\[1ex]
In this section, we first determine the quiver of the algebra $B_\mathcal{Q}$ and compute some concrete examples of $B_\mathcal{Q}$ and of the functor $\Lambda$. We explain the relation of $B_\mathcal{Q}$ to the Auslander algebra of $k\mathcal{Q}$ and give a characterization of the essential image of $\Lambda$.

\subsection{Quiver of $B_\mathcal{Q}$}

We are now able to compute the (ordinary) quiver of the algebra $B_\mathcal{Q}$:

\begin{prop}
The quiver $\widehat{\mathcal{Q}}$ of the algebra $B_\mathcal{Q}$ is given as follows: it has vertices $[U]$ parametrized by isomorphism classes of non-projective indecomposables in $\bmod\, k\mathcal{Q}$, together with vertices $[i]$ for $i\in \mathcal{Q}_0$. There is an arrow $[U]\rightarrow [V]$ for every irreducible map $V\rightarrow U$ between non-projective indecomposables. Moreover, there are arrows $ [i]\rightarrow[S_i]$, resp. $[\tau^{-1}S_i]\rightarrow[i]$, for every vertex $i\in \mathcal{Q}_0$, as long as $S_i$ is non-projective, resp.~non-injective.
\end{prop}

\begin{proof} Using the above projective resolutions of the simple functors, we can compute the
${\rm Ext}$-quiver of the algebra $B_\mathcal{Q}$. Namely, we can compute ${\rm Ext}^1(S_{P_U\subset Q_U},F)$ as the first homology of the complex with terms ${\rm Hom}({\rm Hom}(\_,(P_X\subset Q_X)),F)$ with $X$ being $\tau U$ or $B$ or $U$, respectively, which using Yoneda simplifies to the complex
$$F(P_{U}\subset Q_{U})\rightarrow F(P_B\subset Q_B)\rightarrow F(P_{\tau U}\subset Q_{\tau U}).$$
Now suppose that ${\rm Ext}^1(S_{P_U\subset Q_U},S_{P_V\subset Q_V})$ is non-zero. Then $S_{P_V\subset Q_V}(P_B\subset Q_B)$ is non-zero, thus $V$ is a direct summand of $B$. But then $V$ fulfills the following: it admits an irreducible map to $U$ in $\bmod\, k\mathcal{Q}$, it occurs as a direct summand of $B$ with multiplicity one, and it is not a direct summand of $U$ or of $\tau U$. This in turn implies that ${\rm Ext}^1(S_{P_U\subset Q_U},S_{P_V\subset Q_V})$ is one-dimensional. We have thus proved that $ {\rm Ext}^1(S_{P_U\subset Q_U},S_{P_V\subset Q_V})\not=0$ if and only if $V$ admits an irreducible map to $U$, in which case ${\rm Ext}^1(S_{P_U\subset Q_U},S_{P_V\subset Q_V})$ is one-dimensional.

Similarly we compute ${\rm Ext}^1(S_{P_i=P_i},F)$ as the first homology of the complex
$$0\rightarrow F(P_i=P_i)\rightarrow F(\bigoplus_{i\rightarrow j}P_j\subset P_i),$$
which for $F=S_{P_U\subset Q_U}$ is obviously non-zero (and one-dimensional in this case) if and only if $U=S_i$.

It also follows that ${\rm Ext}^1(S_{P_i=P_i},S_{P_j=P_j})=0$ for all $i,j\in \mathcal{Q}_0$.

Finally, to compute ${\rm Ext}^1(S_{P_U\subset Q_U},S_{P_i=P_i})$, we use an injective coresolution of $S_{P_i=P_i}$ analogous to the projective resolution exhibited in the proof of Theorem \ref{gldim2}. We use the injective coresolution
$$0\rightarrow S_i\rightarrow I_i\rightarrow I_i/{\rm soc}I_i\simeq\bigoplus_{j\rightarrow i}I_j
\rightarrow 0,$$
which yields a projective resolution
$$0\rightarrow P_i\rightarrow\bigoplus_{j\rightarrow i}P_j\rightarrow\tau^{-1}S_i\rightarrow 0$$
using the inverse Nakayama functor. From this, we can easily derive the injective coresolution
$$0\rightarrow S_{P_i=P_i}\rightarrow{\rm Hom}((P_i=P_i),\_)^*\rightarrow{\rm Hom}((P_i\subset\bigoplus_{j\rightarrow i}P_j),\_)^*\rightarrow 0.$$
Similarly to the above, we see that ${\rm Ext}^1(S_{P_U\subset Q_U},S_{P_i=P_i})$ is non-zero (and one-dimensional in this case) if and only if $U\simeq\tau^{-1}S_i$. The theorem is proved.
\end{proof}
\subsection{Examples} 
We now give some examples of the quivers $\widehat{\mathcal{Q}}$ and of their representations $\widehat{M}$.

\begin{example}\label{Ex:An}
Let $\mathcal{Q}$ be the equioriented quiver of type $A_n$. Then the quiver $\widehat{\mathcal{Q}}$ of $B_\mathcal{Q}$
is the Auslander-Reiten quiver of $k\mathcal{Q}$, and the algebra $B_\mathcal{Q}$ is given by imposing all commutativity relations, but no zero relations (see subsection \ref{eAn} for more details). We want to stress that in general the quiver $\widehat{\mathcal{Q}}$ {\bf does not} coincide
with the AR quiver of $\mathcal{Q}$ and the algebra $B_\mathcal{Q}$ is {\bf not} isomorphic to the
Auslander algebra of $\mathcal{Q}$.
\end{example}
\begin{example}\label{Ex:A3}
Let $\mathcal{Q}$ be $\xymatrix@1{1\ar[r]&2\ar[r]&3}$, the equioriented quiver of type $A_3$.
The algebra $B_\mathcal{Q}$ is given by the following quiver with one commutativity relation
$$
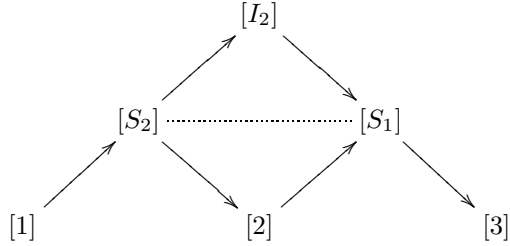
\xymatrix{
                 &                                    &[I_2]\ar[dr]&                 &\\
                 &[S_2]\ar@{..}[rr]\ar[ur]\ar[dr]&                 &[S_1]\ar[dr]&\\
[1]\ar[ur]&                                    &[2]\ar[ur]&                 &[3]
}
$$
\end{example}
\begin{example}\label{Ex:A3noequi}
Let $\mathcal{Q}$ be a quiver $\xymatrix@1{1\ar[r]&2&\ar[l]3}$ of type $A_3$. The algebra $B_\mathcal{Q}$ is given by the following quiver of type $E_6$
$$
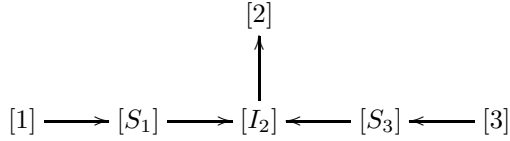
\xymatrix{
            &              &[2]           &              &\\
[1]\ar[r]&[S_1]\ar[r]&[I_2]\ar[u]&[S_3]\ar[l]&[3]\ar[l]
}
$$
\end{example}
\begin{example}\label{Ex:D4}
Let $\mathcal{Q}$ be $$
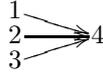
\xymatrix@1@R=1pt{&1\ar[dr]&&\\&2\ar[r]&4\\&3\ar[ur]&}$$ which is the "three subspaces" quiver of type $D_4$. The algebra $B_\mathcal{Q}$ is given by the following quiver with four mesh relations
$$
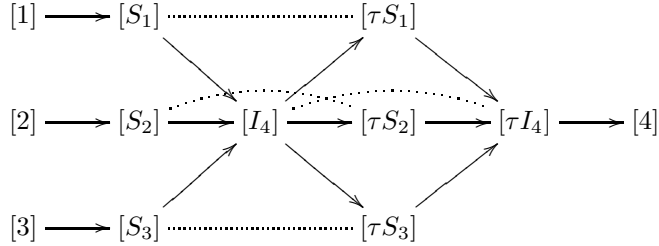
\xymatrix{
[1]\ar[r]                &[S_1]\ar[dr]\ar@{..}[rr]&                     &[\tau S_1]\ar[dr]&        &\\
[2]\ar[r] &  [S_2]\ar[r]\ar@/^1pc/@{..}[rr]       &[I_4]\ar[ur]\ar[dr]\ar[r]\ar@/^1pc/@{..}[rr]&[\tau S_2]\ar[r]         &[\tau I_4]\ar[r]&[4]\\
[3]\ar[r]&                [S_3]\ar[ur]\ar@{..}[rr]&                     &[\tau S_3]\ar[ur]&        &
}
$$
\end{example}

We now give examples of the functor $\Lambda:\mathrm{mod }k\mathcal{Q}\rightarrow \mathrm{mod }B_\mathcal{Q}:$
$M\mapsto\widehat{M}$.

\begin{example}\label{ExA3Mhat}
Let $M:=\xymatrix@1{M_1\ar^f[r]&M_2\ar[r]^g&M_3}$ be a (finite dimensional) representation of the equioriented quiver of type $A_3$. The algebra $B_\mathcal{Q}$ is shown in Example \ref{Ex:A3}. The $B_\mathcal{Q}$-module $\widehat{M}$ is the following
$$
\xymatrix{
                 &                                    &\text{Im }g\circ f\ar@<-1ex>@{^{(}->}[dr]&                 &\\
                 &\text{Im }f\ar@{->>}[ur]^g\ar@{^{(}->}[dr]&                 &\text{ Im }g\text{ }\ar@<-1ex>@{^{(}->}[dr]&\\
M_1\ar@{->>}[ur]^f&                                    &M_2\ar@{->>}[ur]^g&                 &M_3
}
$$
\end{example}
\begin{example}\label{Ex:A3noequiMhat}
Let $M:=\xymatrix@1{M_1\ar^f[r]&M_2&\ar_g[l]M_3}$ be a (finite dimensional) representation of the quiver $\mathcal{Q}:=\xymatrix@1{1\ar[r]&2&\ar[l]3}$. The algebra $B_\mathcal{Q}$ is shown in Example \ref{Ex:A3noequi}. The $B_\mathcal{Q}$--module $\widehat{M}$ is the following
$$
\xymatrix{
            &              &M_2           &              &\\
M_1\ar@{->>}[r]^f&\text{ Im }f\text{ }\ar@{^{(}->}[r]&*+<1pc>{\text{ Im }[f,g]}\ar@{^{(}->}[u]&\text{ Im }g\text{ }\ar@{_{(}->}[l]&M_3\ar@{->>}_g[l]
}
$$
where $[f,g]:M_1\oplus M_3\rightarrow M_2:$ $(v,w)\mapsto f(v)+g(w)$.
\end{example}
\begin{example}\label{Ex:D4Mhat}
Let $\xymatrix@1@R=1pt{M:=&M_1\ar[dr]^{f_1}&&&\\&M_2\ar[r]_{f_2}&M_4&M_3\ar[l]_{f_3}}$ be a finite dimensional representation of the quiver of Example \ref{Ex:D4}. The $B_\mathcal{Q}$--module $\widehat{M}$ is given by
$$
\xymatrix@C=16pt{
M_1\ar@{->>}[r]^{f_1}                &[f_1]\ar@{^{(}->}[dr]^(.4){\left[\tiny{\txt{1\\0}}\right]}&                     &*+<1pc>{[f_2,f_3]}\ar@{^{(}->}[dr]&        &\\
M_2\ar@{->>}[r]^{f_2} &[f_2]\ar@{^{(}->}[r]^(.3){\left[\tiny{\txt{0\\1}}\right]}       &*+<1pc>{\left[\begin{array}{ccc}f_1&0&f_3\\0&f_2&-f_3\end{array}\right]}\ar@{->>}[ur]^{[0,1]}\ar@{->>}[dr]_(.6){[-1,-1]}
\ar@{->>}[r]^(.6){[1,0]}&*+<1pc>{[f_1,f_3]}\ar@{^{(}->}[r]         &*+<1pc>{[f_1,f_2,f_3]}\ar@{^{(}->}[r]&M_4\\
M_3\ar@{->>}[r]^{f_3}&[f_3]\ar@{^{(}->}[ur]_(.4){\left[\tiny{\txt{1\\-1}}\right]}&                     &*+<1pc>{[f_1,f_2]}\ar@{^{(}->}[ur]&        &
}
$$
In the picture above we use the following convention: for a linear map $f:V\rightarrow W$ we denote by $[f]:=\text{ Im }f\subset W$. For example $\left[\begin{array}{ccc}f_1&0&f_3\\0&f_2&-f_3\end{array}\right]$ denotes the image of the linear map $\left[\begin{array}{ccc}f_1&0&f_3\\0&f_2&-f_3\end{array}\right]:M_1\oplus M_2\oplus M_3\rightarrow M_4\oplus M_4$.
\end{example}

\subsection{Comparison between $B_\mathcal{Q}$ and $A_\mathcal{Q}$}
We discuss the relation between the algebra $B_\mathcal{Q}$ and the Auslander algebra $A_\mathcal{Q}=B(\bmod\, k\mathcal{Q})$ of $\bmod\, k\mathcal{Q}$. The equivalence $\mathcal{H}_\mathcal{Q}/{\rm Hom}^{\rm iso}({\rm proj}\, k\mathcal{Q})\simeq\underline{\bmod}\, k\mathcal{Q}$ immediately identifies a quotient of $B_\mathcal{Q}$ with the subalgebra ${\rm End}(\bigoplus_UU)^{\rm op}$ of $A_\mathcal{Q}$, where the sum runs over all non-projective indecomposables of $\bmod\, k\mathcal{Q}$. Moreover, $A_\mathcal{Q}$ arises via tilting (see \cite[VI.]{ASS}) from $B_\mathcal{Q}$:

\begin{prop} The direct sum $T=\bigoplus_U\widehat{U}$ over  indecomposable modules $U$ in $\bmod\, k\mathcal{Q}$
is a tilting object in $\bmod\,\mathcal{H}_\mathcal{Q}^{\rm op}$, such that ${\rm End}(T)^{\rm op}$ is isomorphic to the Auslander algebra $A_\mathcal{Q}$ of $k\mathcal{Q}$.
\end{prop}

\begin{proof} Theorem \ref{ext2ext1} shows that ${\rm Ext}^1(T,T)=0$, and that $T$ has projective dimension at most $1$. Rewriting the projective resolution of $\widehat{M}$ of the proof of Theorem \ref{ext2ext1} as
$$0\rightarrow {\rm Hom}(\_,(P\subset Q))\rightarrow\widehat{Q}\rightarrow\widehat{M}\rightarrow 0$$
for a projective resolution $0\rightarrow P\rightarrow Q\rightarrow M\rightarrow 0$, we see that all indecomposable projective objects in $\bmod\,\mathcal{H}_\mathcal{Q}^{\rm op}$ admit a short coresolution by sums of direct summands of $T$, proving that $T$ is a tilting object. The functor $\Lambda$ being fully faithful, we see that $${\rm End}(T)^{\rm op}\simeq{\rm End}(\bigoplus_UU)^{\rm op}=A_\mathcal{Q}.$$
\end{proof}
The tilting $B_\mathcal{Q}$--module $T$ induces a derived equivalence $F:=RHom_{B_\mathcal{Q}}(T,\_):\mathcal{D}^b(\text{mod--}B_\mathcal{Q})\rightarrow \mathcal{D}^b(\text{mod--}A_\mathcal{Q})$ \cite[Theorem 1.6]{H}. The image $F(B_\mathcal{Q})$ of $B_\mathcal{Q}$ under this functor is a tilting complex $T'$ whose endomorphism ring is the algebra $B_\mathcal{Q}$. The next proposition shows this tilting complex explicitly. In order to state the result we need a little preparation. Under the isomorphism $\text{End}_{B_\mathcal{Q}}(T)^{\text{op}}\simeq A_\mathcal{Q}$, every direct summand
$\widehat{M}$ for $M\in \textrm{ind }k\mathcal{Q}$ of $T$ corresponds to an indecomposable projective $A_\mathcal{Q}$--module which we denote by $A_{M}$.
\begin{prop}
The tilting complex $T'$ in $\mathcal{D}^b(\text{mod--}A_\mathcal{Q})$ is given as follows
$$
T'={\displaystyle \bigoplus_{i\in \mathcal{Q}_0}}A_{P_i}\oplus\bigoplus_{U\in \textrm{ind }k\mathcal{Q}\setminus
\textrm{proj }k\mathcal{Q}} (A_{Q_U}\rightarrow A_{U})
$$
where $0\rightarrow P_U\rightarrow Q_U\rightarrow U\rightarrow 0$ is the minimal projective resolution of $U$, the complex $A_{P_i}$ is concentrated in degree $0$, and the complex $(A_{Q_U}\rightarrow A_U)$ is concentrated in degrees $0$ and $1$.
\end{prop}
\begin{proof}
We write $B_\mathcal{Q}=\bigoplus_{i} B_i$ as the sum of all the indecomposable projective $B_\mathcal{Q}$--modules. Recall that these modules correspond to the functors $\text{Hom}(\_,(P_i=P_i))$, $i\in \mathcal{Q}_0$ and $\text{Hom}(\_,(P_U\rightarrow  Q_U))$, for $U\in \textrm{ind }k\mathcal{Q}\setminus \textrm{proj }k\mathcal{Q}$. Since
$\text{Hom}(\_, (P_i=P_i))\simeq \widehat{P_i}$, the corresponding $B_i$ is a summand of $T$ and hence $F(B_i)=A_{P_i}$ is an indecomposable projective $A_\mathcal{Q}$--module. It remains to find the image of the remaining direct summands
$B_j=\text{Hom}(\_,(P_U\rightarrow  Q_U))$ of $B$, for $U\in \textrm{ind }k\mathcal{Q}\setminus \textrm{proj }k\mathcal{Q}$.
Every such projective $B_j$ arises in the projective resolution of $B_\mathcal{Q}$--modules
$$
0\rightarrow B_j \rightarrow\widehat{Q_U}\rightarrow \widehat{U}\rightarrow 0.
$$
This induces a triangle
$$
B_j\rightarrow \widehat{Q_U}\rightarrow \widehat{U}\rightarrow B_j[1]
$$
in $\mathcal{D}^b(B_\mathcal{Q})$. We apply the triangle functor $F$ to this triangle and we get a triangle
$$
F(B_j)\rightarrow A_{Q_U}\rightarrow A_{U}\rightarrow F(B_j)[1]
$$
in $\mathcal{D}^b(A_\mathcal{Q})$. From this triangle we conclude that $F(B_j)$ is isomorphic to the complex
$(A_{Q_U}\rightarrow A_{U})$ (in degrees $0$ and $1$), as desired.
\end{proof}

\subsection{Essential image of $\Lambda$}
In Examples \ref{ExA3Mhat}, \ref{Ex:A3noequiMhat} and \ref{Ex:D4Mhat} we see that all the linear maps of the $\widehat{\mathcal{Q}}$--representation $\widehat{M}$ are either injective or surjective. The next proposition
shows that such properties hold in general, encoded in the vanishing of certain homomorphism spaces. In fact, we can give a characterization of the essential image of the functor $\Lambda$ as follows:

\begin{prop}\label{Prop:F=Mhat}
A functor $F\in\bmod\,\mathcal{H}_\mathcal{Q}^{\rm op}$ is isomorphic to $\widehat{M}$ for some $M\in\bmod\, k\mathcal{Q}$ if and only if ${\rm Hom}(S_{P_U\subset Q_U},F)=0={\rm Hom}(F,S_{P_U\subset Q_U})$ for all non-projective indecomposables $U\in\bmod\, k\mathcal{Q}$.
\end{prop}

\begin{proof} We have ${\rm res}\, S_{P_U\subset Q_U}=0$ by definition. Lemma \ref{lemweak} implies
$${\rm Hom}(\widehat{M},S_{P_U\subset Q_U})\subset {\rm Hom}(M,{\rm res}\, S_{P_U\subset Q_U})=0,$$
$${\rm Hom}(S_{P_U\subset Q_U},\widehat{M})\subset {\rm Hom}({\rm res}\, S_{P_U\subset Q_U},M)=0.$$
To prove the converse, assume that ${\rm Hom}(S_{P_U\subset Q_U},F)=0={\rm Hom}(F,S_{P_U\subset Q_U})$ for all non-projective indecomposables $U\in\bmod\, k\mathcal{Q}$ for a functor $F$. We define $M={\rm res}\, F$ and have to prove that $F\simeq \widehat{M}$. By definition, this amounts to proving the following: given an object $P\subset Q$ in $\mathcal{H}_\mathcal{Q}$, we have canonical maps $(P=P)\rightarrow (P\subset Q)\rightarrow (Q=Q)$ in $\mathcal{H}_\mathcal{Q}$ inducing a sequence
$$F(Q=Q)\rightarrow F(P\subset Q)\rightarrow F(P=P).$$
Then we have to prove that the first map is surjective and the second map is injective. We prove injectivity of the second map; surjectivity of the first map is proved dually. First we can restrict to the case of $P\subset Q$ being indecomposable, thus $(P\subset Q)=(P_U\subset Q_U)$ for a non-projective indecomposable $U$. Assume that $U$ is such that there exists an element $0\not=x\in F(P_U\subset Q_U)$ mapping to zero in $F(P_U=P_U)$. Without loss of generality, we can assume $U$ to be minimal with this property with respect to the ordering induced by irreducible maps. Using the description $$S_{P_U\subset Q_U}\simeq{\rm Hom}(\_,(P_U\subset Q_U))/{\rm rad}\,{\rm Hom}(\_,(P_U\subset Q_U)),$$
we can rewrite ${\rm Hom}(S_{P_U\subset Q_U},F)$ as the intersection of the kernels of the maps $F(f)$ for $f$ ranging over the non-split maps $f:(P_V\subset Q_V)\rightarrow(P_U\subset Q_U)$ in $\mathcal{H}_\mathcal{Q}$. Since this intersection is zero by assumption, there exists an indecomposable object $(P_V\subset Q_V)$ and a non-split map $f:(P_V\subset Q_V)\rightarrow (P_U\subset Q_U)$ such that $F(f)(x)\not=0$. We have a natural square
$$\begin{array}{ccc}
(P_V=P_V)&\rightarrow&(P_V\subset Q_V)\\
\downarrow&&\downarrow\\
(P_U=P_U)&\rightarrow&(P_U\subset Q_U)\end{array}$$
inducing the square
$$\begin{array}{ccc}
F(P_V=P_V)&\leftarrow&F(P_V\subset Q_V)\\
\uparrow&&\uparrow\\
F(P_U=P_U)&\leftarrow&F(P_U\subset Q_U).\end{array}$$
The element $x$ mapping to zero under the lower horizontal map, we see that $F(f)(x)\not=0$ maps to zero under the upper horizontal map, a contradiction to the minimality of $U$. The proposition is proved.
\end{proof}
The following examples show the AR--quiver of some algebras $B_\mathcal{Q}$ of finite representation type. These pictures also illustrate the statement of the previous proposition.

\begin{example}
Let $\mathcal{Q}:=\xymatrix@1{1\ar[r]&2\ar[r]&3}$ be the quiver of type $A_3$ already considered in Examples \ref{Ex:A3} and \ref{ExA3Mhat}. From the description of $B_\mathcal{Q}$ given in Example \ref{Ex:A3}, it follows that $B_\mathcal{Q}$ is of finite representation type. The following quiver is the AR--quiver of $B_\mathcal{Q}$.
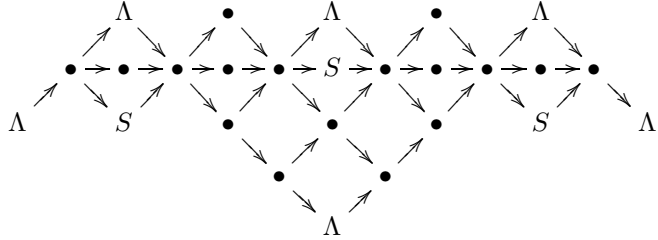
$$
\xymatrix@R=8pt@C=8pt{
&&\Lambda\ar[dr]&&\bullet\ar[dr]&&\Lambda\ar[dr]&&\bullet\ar[dr]&&\Lambda\ar[dr]&&\\
&\bullet\ar[ur]\ar[r]\ar[dr]&\bullet\ar[r]&\bullet\ar[ur]\ar[r]\ar[dr]&\bullet\ar[r]&\bullet\ar[ur]\ar[r]\ar[dr]&S\ar[r]&\bullet\ar[ur]\ar[r]\ar[dr]&\bullet\ar[r]&\bullet\ar[ur]\ar[r]\ar[dr]&\bullet\ar[r]&\bullet\ar[dr]&\\
\Lambda\ar[ur]&&S\ar[ur]&&\bullet\ar[ur]\ar[dr]&&\bullet\ar[ur]\ar[dr]&&\bullet\ar[ur]&&S\ar[ur]&&\Lambda\\
&&&&&\bullet\ar[ur]\ar[dr]&&\bullet\ar[ur]&&&&&\\
&&&&&&\Lambda\ar[ur]&&&&&&
}
$$
In the picture above we denote by $\Lambda$ the vertices corresponding to the $B_\mathcal{Q}$--modules $\widehat{U}$, for $U\in\textrm{ind }k\mathcal{Q}$. We denote by $S$ the vertices corresponding to the simple $B_\mathcal{Q}$--modules $S_{P_U\subset Q_U }$, $U\in\textrm{ind }k\mathcal{Q}\setminus \textrm{proj }k\mathcal{Q}$.
\end{example}

\begin{example}
Let $\mathcal{Q}:=\xymatrix@1{1\ar[r]&2&3\ar[l]}$ be the quiver of type $A_3$ already considered in Examples \ref{Ex:A3noequi} and \ref{Ex:A3noequiMhat}. From the description of $B_\mathcal{Q}$ given in Example \ref{Ex:A3noequi}, it follows that $B_\mathcal{Q}$ is of finite representation type. The following quiver is the AR--quiver of $B_\mathcal{Q}$ (which is of type $E_6$).
$$
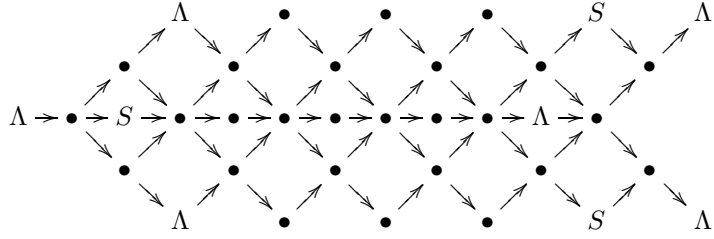
\xymatrix@R=8pt@C=8pt{
&&&\Lambda\ar[dr]&&\bullet\ar[dr]&&\bullet\ar[dr]&&\bullet\ar[dr]&&S\ar[dr]&&\Lambda&\\
&&\bullet\ar[ur]\ar[dr]&&\bullet\ar[ur]\ar[dr]&&\bullet\ar[ur]\ar[dr]&&\bullet\ar[ur]\ar[dr]&&\bullet\ar[ur]\ar[dr]&&\bullet\ar[ur]&\\
\Lambda\ar[r]&\bullet\ar[ur]\ar[dr]\ar[r]&S\ar[r]&\bullet\ar[ur]\ar[dr]\ar[r]&\bullet\ar[r]&\bullet\ar[ur]\ar[dr]\ar[r]&\bullet\ar[r]&\bullet\ar[ur]\ar[dr]\ar[r]&\bullet\ar[r]&\bullet\ar[ur]\ar[dr]\ar[r]&\Lambda\ar[r]&\bullet\ar[ur]\ar[dr]&&\\
&&\bullet\ar[ur]\ar[dr]&&\bullet\ar[ur]\ar[dr]&&\bullet\ar[ur]\ar[dr]&&\bullet\ar[ur]\ar[dr]&&\bullet\ar[ur]\ar[dr]&&\bullet\ar[dr]&\\
&&&\Lambda\ar[ur]&&\bullet\ar[ur]&&\bullet\ar[ur]&&\bullet\ar[ur]&&S\ar[ur]&&\Lambda&\\}
$$
We use the notation $\Lambda$ and $S$ in the same way as in the previous example.
\end{example}

\section{Construction of the desingularization}\label{section7}

Now we assume $k$ to be algebraically closed.

\begin{prop}\label{grsmooth} Let $\mathcal{Q}$ be a quiver (not necessarily of Dynkin type), let $M$ be a representation of $k\mathcal{Q}$ and let ${\bf e}$ be a dimension vector for $\mathcal{Q}$. Assume that $M$ is a representation for a quotient algebra $A=k\mathcal{Q}/I$, that is, the two-sided ideal $I$ annihilates $M$, such that the following holds: $A$ has global dimension at most two, both the injective and the projective dimension of $M$ over $A$ are at most one, and ${\rm Ext}^1_A(M,M)=0$. Then the quiver Grassmannian ${\rm Gr}_{\bf e}(M)$ is smooth (and reduced), with irreducible and equidimensional connected components.
\end{prop}

\begin{proof} Let $\langle\_,\_\rangle_A$ be the homological Euler form of $A$, thus
$$\dim{\rm Hom}_A(X,Y)-\dim{\rm Ext}^1_A(X,Y)+\dim{\rm Ext}_A^2(X,Y)=\langle\dimv X,\dimv Y\rangle_A$$
for all representations $X$ and $Y$ of $A$. We first estimate the dimension of ${\rm Gr}_{\bf e}(M)$ by generalizing the arguments of \cite[Section 2, Proposition 2.2]{CFR}. For ${\bf d}$ the dimension vector of $M$, let $R_{\bf d}(A)$ be the closed subset of the representation variety $R_{\bf d}(\mathcal{Q})=\prod_{\alpha:i\rightarrow j}{\rm Hom}(M_i,M_j)$ consisting of representations annihilated by $I$. We consider the product of ordinary Grassmannians ${\rm Gr}_{\bf e}({\bf d})=\prod_{i\in\mathcal{Q}_0}{\rm Gr}_{e_i}(M_i)$ and define the universal quiver Grassmannian ${\rm Gr}_{\bf e}^A({\bf d})$ as the closed subset of ${\rm Gr}_{\bf e}({\bf d})\times R_{\bf d}(A)$ consisting of pairs $((U_i)_i,(f_\alpha)_\alpha)$ such that $f_\alpha(U_i)\subset U_j$ for all arrows $\alpha:i\rightarrow j$ in $\mathcal{Q}$. Using the projection $p_2:{\rm Gr}_{\bf e}^A({\bf d})\rightarrow R_{\bf d}(A)$, we define the scheme-theoretic quiver Grassmannian $\mathcal{G}r_{\bf e}(M)=p_2^{-1}(M)$, thus by definition, ${\rm Gr}_{\bf e}(M)$ is isomorphic to $\mathcal{G}r_{\bf e}(M)$ endowed with the reduced structure. On the other hand, we consider the projection $p_1:{\rm Gr}_{\bf e}^A({\bf d})\rightarrow{\rm Gr}_{\bf e}({\bf d})$. Its fibres can be identified with the subvariety $Z$ of $R_{\bf d}(A)$ consisting of representations in upper-triangular block form
$$(\left[\begin{array}{rr}g_\alpha&\zeta_\alpha\\ 0&h_\alpha\end{array}\right])_\alpha.$$
As in \cite[Section 2.1]{BoMin}, we have a projection $q:Z\rightarrow R_{\bf e}(A)\times R_{{\bf d}-{\bf e}}(A)$ by restricting to the diagonal blocks, such that the fibre over a point of $R_{\bf e}(A)\times R_{{\bf d}-{\bf e}}(A)$ corresponding to a pair of representations $(N,Q)$ is a vector space of dimension
$$\sum_{i\in\mathcal{Q}_0}e_i(d_i-e_i)-\dim {\rm Hom}_A(Q,N)+\dim{\rm Ext}^1_A(Q,N)=$$
$$=\sum_ie_i(d_i-e_i)-\langle {\bf d}-{\bf e},{\bf e}\rangle_A+\dim{\rm Ext}^2_A(Q,N).$$
In particular, every fibre of $q$ has at least dimension $\sum_ie_i(d_i-e_i)-\langle{\bf d}-{\bf e},{\bf e}\rangle_A$.
By \cite[Section 2.1]{BoQ}, every irreducible component of the variety $R_{\bf e}(A)$ has at least dimension $\sum_{i\in\mathcal{Q}_0}e_i^2-\langle{\bf e},{\bf e}\rangle_A$, and similarly for $R_{{\bf d}-{\bf e}}(A)$, yielding an estimate for the dimension of $Z$. The resulting estimate for the dimension of the universal quiver Grassmannian ${\rm Gr}_{\bf e}^A({\bf d})$ can be calculated as
$$\sum_id_i^2-\langle {\bf d},{\bf d}\rangle_A+\langle {\bf e},{\bf d}-{\bf e}\rangle_A.$$
By the assumptions on $M$, its corresponding orbit in $R_{\bf d}(A)$ is open of dimension $\sum_id_i^2-\langle {\bf d},{\bf d}\rangle_A$, and thus every irreducible component of the quiver Grassmannian ${\rm Gr}_{\bf e}(M)$ has at least dimension $\langle{\bf e},{\bf d}-{\bf e}\rangle_A$.\\[1ex]
Now we prove that the dimension of the tangent space $T_N(\mathcal{G}r_{\bf e}(M))$ equals $\langle {\bf e},{\bf d}-{\bf e}\rangle_A$ for all subrepresentations $N$ of dimension vector ${\bf e}$, which, together with the above dimension estimate, proves reducedness, smoothness, equidimensionality and the fact that no two irreducible components intersect. The tangent space in question can be identified with ${\rm Hom}_A(N,M/N)$ by \cite{CFR}. Using the formula
$$\dim{\rm Hom}_A(N,M/N)-\dim{\rm Ext}^1_A(N,M/N)+\dim{\rm Ext}_A^2(N,M/N)=\langle{\bf e},{\bf d}-{\bf e}\rangle_A, $$
we are thus finished once we can prove that ${\rm Ext}^1_A(N,M/N)=0={\rm Ext}^2_A(N,M/N)$. Applying ${\rm Hom}_A(N,\_)$ to the exact sequence $0\rightarrow N\rightarrow M\rightarrow M/N\rightarrow 0$ and working out the resulting long exact sequence, we see that
$$0={\rm Ext}_A^2(N,M)\rightarrow{\rm Ext}_A^2(N,M/N)\rightarrow{\rm Ext}^3_A(N,N)=0,$$
thus ${\rm Ext}_A^2(N,M/N)=0$. Working out various other long exact cohomology sequences, we get the following commutative diagram with exact rows and columns:
$$\begin{array}{ccccc}
{\rm Ext}^1_A(M,M)&\rightarrow{\rm Ext}_A^1(M,M/N)&\rightarrow&{\rm Ext}^2_A(M,N)\\
\downarrow&\downarrow&&\downarrow\\
{\rm Ext}^1_A(N,M)&\rightarrow{\rm Ext}_A^1(N,M/N)&\rightarrow&{\rm Ext}^2_A(N,N)\\
\downarrow&\downarrow&&\downarrow\\
{\rm Ext}^2_A(M/N,M)&\rightarrow{\rm Ext}_A^2(M/N,M/N)&\rightarrow&{\rm Ext}^3_A(M/N,N)
\end{array}$$
By assumption, all four corners of this square are zero, thus the central term is zero, proving ${\rm Ext}_A^1(N,M/N)=0$.
\end{proof}

For a given representation $M$, a dimension vector ${\bf e}$ and an isomorphism class $[N]$ for a $\mathcal{Q}$, it is proved in \cite[Section 2.3]{CFR} that the subset $\mathcal{S}_{[N]}$ of ${\rm Gr}_{\bf e}(M)$, consisting of the subrepresentations which are isomorphic to $N$, is locally closed and irreducible of dimension $\dim{\rm Hom}(N,M)-\dim{\rm End}(N)$.

\begin{prop}\label{strataclosure} Suppose that $\mathcal{S}_{[U]}$ has non-empty intersection with $\overline{\mathcal{S}_{[N]}}$. Then $\dimv\widehat{U}\leq\dimv\widehat{N}$ componentwise as dimension vectors of representations of $B_\mathcal{Q}$.
\end{prop}

\begin{proof} First we briefly recall the geometric definition of the above strata: we consider the variety $R_{\bf e}(\mathcal{Q})$ of representations of $\mathcal{Q}$ of dimension vector ${\bf e}$, with its standard base change action of the group $G_{\bf e}$, such that the orbits $\mathcal{O}_{[N]}$ correspond bijectively to the isomorphism classes $[N]$ of representations of $\mathcal{Q}$ of dimension vector ${\bf e}$. There exists a locally trivial $G_{\bf e}$-principal bundle $\pi:X_{\bf e}(M)\rightarrow{\rm Gr}_{\bf e}(M)$ which admits a $G_{\bf e}$-equivariant map $p:X_{\bf e}(M)\rightarrow R_{\bf e}(\mathcal{Q})$. The stratum $\mathcal{S}_{[N]}$ is then defined by $\pi(p^{-1}(\mathcal{O}_{[N]}))$.

Now suppose that $\mathcal{S}_{[U]}$ has non-empty intersection with $\overline{\mathcal{S}_{[N]}}$. Using the above definitions and the fact that $\pi$ is a $G_{\bf e}$-principal bundle, this means that $\pi(p^{-1}(\mathcal{O}_{[U]}))$ has non-empty intersection with
$$\overline{\pi(p^{-1}(\mathcal{O}_{[N]}))}\subset \pi(\overline{p^{-1}(\mathcal{O}_{[N]})}).$$
Thus $p^{-1}(\mathcal{O}_{[U]})$ has non-empty intersection with $\overline{p^{-1}(\mathcal{O}_{[N]})}$, which implies that $\mathcal{O}_{[U]}=p(p^{-1}(\mathcal{O}_{[U]}))$ has non-empty intersection with $$p(\overline{p^{-1}(\mathcal{O}_{[N]})})\subset p(p^{-1}(\overline{\mathcal{O}_{[N]}}))\subset\overline{\mathcal{O}_{[N]}},$$
and thus is already contained in $\overline{\mathcal{O}_{[N]}}$. This adherence relation $\mathcal{O}_{[U]}\subset\overline{\mathcal{O}_{[N]}}$ now implies that $\dim {\rm Hom}(P,U)=\dim{\rm Hom}(P,N)$ for all projective representations $P$, and $\dim{\rm Hom}(X,U)\geq\dim{\rm Hom}(X,N)$ for all non-projectives $X$ (see \cite{Bo}). Now consider the dimension vector of $\widehat{U}$, resp.~of $\widehat{N}$, as a representation of $B_\mathcal{Q}$, thus $(\dimv\widehat{U})_{[X]}=\dim \widehat{U}(P_X\subset Q_X)$ and $(\dimv\widehat{U})_{[i]}=\dim \widehat{U}(P_i=P_i)=(\dimv U)_i$ as above. Using the exact sequence
$$0\rightarrow{\rm Hom}(X,U)\rightarrow{\rm Hom}(Q_X,U)\rightarrow{\rm Hom}(P_X,U),$$
we can calculate
$$(\dimv\widehat{U})_{[X]}=\dim \widehat{U}(P_X\subset Q_X)=\dim{\rm Im}({\rm Hom}(Q_X,U)\rightarrow{\rm Hom}(P_X,U))=$$
$$=\dim{\rm Hom}(Q_X,U)-\dim{\rm Hom}(X,U)\leq\dim{\rm Hom}(Q_X,N)-\dim{\rm Hom}(X,N),$$
which in turn equals $(\dimv \widehat{N})_{[X]}$. This proves $\dimv\widehat{U}\leq\dimv\widehat{N}$ componentwise.
\end{proof}

\begin{definition} We call $[N]$ a generic subrepresentation type of $M$ of dimension vector ${\bf e}$ if the stratum $\mathcal{S}_{[N]}$ of ${\rm Gr}_{\bf e}(M)$ is open. Denote by ${\rm gsub}_{\bf e}(M)$ the set of all generic subrepresentation types.
\end{definition}

In case $[N]\in{\rm gsub}_{\bf e}(M)$, the closure $\overline{\mathcal{S}_{[N]}}$ is an irreducible component of ${\rm Gr}_{\bf e}(M)$, and every irreducible component arises in this way.\\[1ex]
For representations $M$ and $N$ of $\mathcal{Q}$, we now consider quiver Grassmannians for the quiver $\widehat{\mathcal{Q}}$ of the form ${\rm Gr}_{\dimv\widehat{N}}(\widehat{M})$. Here $\dimv$ denotes the dimension vector of $\widehat{N}$ as a representation of $\widehat{\mathcal{Q}}$, that is, $(\dimv\widehat{N})_{[U]}=\dim{\rm Im}({\rm Hom}(Q_U,N)\rightarrow{\rm Hom}(P_U,N))$ for all non-projective indecomposables $U$, and $(\dimv\widehat{N})_{[i]}=(\dimv N)_i$ for all $i\in \mathcal{Q}_0$.\\[1ex]
For $[N]\in{\rm gsub}_{\bf e}(M)$, we consider the map
$$\pi_{[N]}:{\rm Gr}_{\dimv\widehat{N}}(\widehat{M})\rightarrow{\rm Gr}_{\bf e}(M)$$
given by $(F\subset\widehat{M})\mapsto({\rm res}\,{F}\subset M)$. Our aim is to construct a desingularization of ${\rm Gr}_{\bf e}(M)$ using the maps $\pi_{[N]}$. We start with a description of the fibres of $\pi_{[N]}$ in terms of quiver Grassmannians.

\begin{prop} For $M$, ${\bf e}$ and $[N]$ as above and a point $(U\subset M)$ in ${\rm Gr}_{\bf e}(M)$, we have an isomorphism
$$
\pi_{[N]}^{-1}(U\subset M)\simeq{\rm Gr}_{\dimv\widehat{N}-\dimv\widehat{U}}(\widehat{M}/\widehat{U}).$$
\end{prop}


\begin{proof} More precisely, we prove that $$\pi_{[N]}^{-1}(U\subset M)\simeq\{F\subset
\widehat{M}\, :\,\dimv F=\dimv \widehat{N},\, \widehat{U}\subset F\}.$$

By definition of the map $\pi_{[N]}$, this immediately reduces to the following statement:\\
Suppose we are given a subrepresentation $U\subset M$ of dimension vector ${\bf e}$ and a subobject $F\subset \widehat{M}$ such that $\dimv F=\dimv\widehat{N}$. Then we have ${\rm res} F=U$ if and only if $\widehat{U}\subset F$.\\
So suppose $\dimv F=\dimv \widehat{N}$ and $\widehat{U}\subset F$. Then $U={\rm res}\,\widehat{U}\subset{\rm res}\, F$ and
$$(\dimv U)_i=\dim{\rm Hom}(P_i,U)=\dim{\rm Hom}(P_i,N)=$$
$$=\dim\widehat{N}(P_i=P_i)=\dim F(P_i=P_i)=(\dimv\,{\rm res}\, F)_i,$$
and thus $U={\rm res}\, F$.

Conversely, suppose that ${\rm res}\, F=U$ and $F\subset\widehat{M}$. For an object $(P\subset Q)$ of $\mathcal{H}_\mathcal{Q}$, the canonical chain of maps $(P=P)\rightarrow(P\subset Q)\rightarrow(Q=Q)$ induces a diagram
$$\begin{array}{ccccc}
F(Q=Q)&\stackrel{\alpha}{\rightarrow}&F(P\subset Q)&\stackrel{\beta}{\rightarrow}&F(P=P)\\
\downarrow&&\downarrow&&\downarrow\\
\widehat{M}(Q=Q)&\rightarrow&\widehat{M}(P\subset Q)&\stackrel{\gamma}{\rightarrow}&\widehat{M}(P=P)\\
||&&||&&||\\
{\rm Hom}(Q,M)&\rightarrow&{\rm Im}({\rm Hom}(Q,M)\rightarrow{\rm Hom}(P,M))&\stackrel{\gamma}{\rightarrow}&{\rm Hom}(P,M).
\end{array}$$
The upper vertical maps being embeddings, and the map $\gamma$ being an
embedding, we see that $\beta$ is an embedding. On the other hand, we have $\dim{\rm Hom}(Q,U)=\dim{\rm Hom}(Q,N)=\dim F(Q=Q)$ and similarly $\dim{\rm Hom}(P,U)=\dim F(P=P)$, which yields a diagram
$$\begin{array}{ccccc}
{\rm Hom}(Q,U)&\rightarrow&{\rm Im}({\rm Hom}(Q,U)\rightarrow{\rm Hom}(P,U))&\rightarrow&{\rm Hom}(P,U)\\
||&&&&||\\
F(Q=Q)&\stackrel{\alpha}{\rightarrow}&F(P\subset Q)&\stackrel{\beta}{\rightarrow}&F(P=P)\end{array}.$$
The upper middle term thus identifies with ${\rm Im}(\beta\alpha)$, whereas the lower middle term identifies with ${\rm Im}(\beta)$ since $\beta$ is an embedding. But then ${\rm Im}(\beta\alpha)$ naturally embeds into ${\rm Im}(\beta)$, thus we have compatible embeddings $\widehat{U}(P\subset Q)\subset F(P\subset Q)$, thus an embedding of functors $\widehat{U}\subset F$ as desired.
\end{proof}

We can now easily derive the main general geometric properties of the map $\pi_{[N]}$:

\begin{thm} For all $M$ and ${\bf e}$ as above and a generic subrepresentation type $[N]\in{\rm gsub}_{\bf e}(M)$, the following holds:
\begin{enumerate}
\item The variety ${\rm Gr}_{\dimv\widehat{N}}(\widehat{M})$ is smooth with irreducible equidimensional connected components,
\item the map $\pi_{[N]}$ is projective,
\item the image of $\pi_{[N]}$ is closed in ${\rm Gr}_{\bf e}(M)$ and contains $\overline{\mathcal{S}_{[N]}}$,
\item the map $\pi_{[N]}$ is one-to-one over $\mathcal{S}_{[N]}$.
\end{enumerate}
\end{thm}

\begin{proof} We apply Proposition \ref{grsmooth} to the quiver $\widehat{\mathcal{Q}}$, the factor algebra $B_\mathcal{Q}$ of $k\widehat{\mathcal{Q}}$ and the representation $\widehat{M}$ of $\widehat{\mathcal{Q}}$, resp.~of $B_\mathcal{Q}$. Then the homological vanishing properties Theorem \ref{gldim2} and Theorem \ref{ext2ext1} imply that the assumptions of Proposition \ref{grsmooth} hold, thus ${\rm Gr}_{\dimv\widehat{N}}(\widehat{M})$ is smooth with irreducible and equidimensional connected components. The map $\pi_{[N]}$ is projective since ${\rm Gr}_{\dimv\widehat{N}}(\widehat{M})$ is projective. Given a generic embedding $N\subset M$, we also have $\widehat{N}\subset\widehat{M}$ since $\Lambda$ is fully faithful, thus the fibre over $N\subset M$ is non-empty. We conclude that the image of $\pi_{[N]}$ contains $\mathcal{S}_{[N]}$, and thus $\overline{\mathcal{S}_{[N]}}$, its image being closed since it is proper. That the fibre over a point of $\mathcal{S}_{[N]}$ reduces to a single point is the special case $U=N$ of the previous proposition.
\end{proof}

All ingredients for the construction of desingularizations are now at hand.\\[1ex] We first treat the case of an irreducible quiver Grassmannian ${\rm Gr}_{\bf e}(M)$. Note that in this case, by the definitions, there exists a unique generic subrepresentation type $[N]$, such that ${\rm Gr}_{\bf e}(M)=\overline{\mathcal{S}_{[N]}}$. The previous theorem immediately implies:

\begin{cor} For $M$ and ${\bf e}$ as above, suppose that the quiver Grassmannian ${\rm Gr}_{\bf e}(M)$ is irreducible, with unique generic subrepresentation type $N$. Then the map $$\pi_{[N]}:{\rm Gr}_{{\dimv}\widehat{N}}(\widehat{M})\rightarrow{\rm Gr}_{\bf e}(M)$$ is a desingularization.
\end{cor}

In the reducible case, we restrict the maps $\pi_{[N]}$ for generic subrepresentation types $[N]$ to the connected components of ${\rm Gr}_{\dimv\widehat{N}}(\widehat{M})$ containing the embeddings of $\widehat{N}$ into $\widehat{M}$:

\begin{cor} For arbitrary $M$ and ${\bf e}$ as above, the closure $\overline{\mathcal{S}_{[\widehat{N}]}}$ is an irreducible component of ${\rm Gr}_{\dimv\widehat{N}}(\widehat{M})$, and the map
$$\pi=\bigsqcup_{[N]\in{\rm gsub}_{\bf e}(M)}\pi_{[N]}:\bigsqcup_{[N]\in{\rm gsub}_{\bf e}(M)}\overline{\mathcal{S}_{[\widehat{N}]}}\rightarrow {\rm Gr}_{\bf e}(M)$$
given by the restrictions of the $\pi_{[N]}$ to $\overline{\mathcal{S}_{[\widehat{N}]}}$ is a desingularization of ${\rm Gr}_{\bf e}(M)$.
\end{cor}

\begin{proof} The stratum $\mathcal{S}_{[\widehat{N}]}$ of ${\rm Gr}_{\dimv\widehat{N}}(\widehat{M})$ is irreducible and locally closed, of dimension 
$$\dim{\rm Hom}(\widehat{N},\widehat{M})-\dim{\rm End}(\widehat{N})=\dim{\rm Hom}(\widehat{N},\widehat{M}/\widehat{N})=\dim{\rm Gr}_{\dimv\widehat{N}}(\widehat{M})$$
as in the proof of Proposition \ref{grsmooth} since ${\rm Ext}^1(\widehat{N},\widehat{N})=0$, thus its closure is a connected (irreducible) component of ${\rm Gr}_{\dimv\widehat{N}}(\widehat{M})$ and thus a smooth variety. The image of this component under $\pi_{[N]}$ is thus an irreducible closed subvariety of ${\rm Gr}_{\bf e}(M)$ containing its irreducible component $\overline{\mathcal{S}_{[N]}}$, and thus it equals $\overline{\mathcal{S}_{[N]}}$.  Together with the other properties of the previous theorem, this implies that $\pi$ is a desingularization.
\end{proof}

\begin{rem} We conjecture that ${\rm Gr}_{\dimv\widehat{N}}(\widehat{M})$ is actually irreducible. This would imply that the constructions of desingularizations of the two previous corollaries could be unified to the map 
$$\pi=\bigsqcup_{[N]\in{\rm gsub}_{\bf e}(M)}\pi_{[N]}:\bigsqcup_{[N]\in{\rm gsub}_{\bf e}(M)}{\rm Gr}_{\dimv\widehat{N}}(\widehat{M})\rightarrow {\rm Gr}_{\bf e}(M)$$
being a desingularization.
\end{rem}

\section{Examples}\label{section8}

\subsection{Equioriented $A_n$ case}\label{eAn}
As the first example, we consider the equioriented type $A_n$ quiver $\mathcal{Q}$ given by $1\rightarrow 2\rightarrow\ldots\rightarrow n$. A representation $M$ is then given by a chain of linear maps
$$M_1\stackrel{f_1}{\rightarrow}M_2\stackrel{f_2}{\rightarrow}\ldots\stackrel{f_{n-1}}{\rightarrow} M_n;$$
a dimension vector ${\bf e}$ is given by a tuple $(e_1,\ldots,e_n)$.\\[1ex]
The indecomposable representations of $\mathcal{Q}$ are the $U_{i,j}$ for $1\leq i\leq j\leq n$ of dimension vector $\dimv U_{i,j}=(\overbrace{0,\dots,0}^{i-1},\overbrace{1,\dots,1}^{j-i+1},0,\dots,0)$.
In particular, we have $P_i=U_{i,n}$, $I_i=U_{1,i}$ and $S_i=U_{i,i}$. The quiver $\widehat{\mathcal{Q}}$ thus has vertices $[i,j]$ for $1\leq i\leq j<n$ and $[i]$ for $1\leq i\leq n$ and the following arrows:
\begin{itemize}
\item $[i,j]\rightarrow[i,j+1]$ for $1\leq i\leq j<n-1$,
\item $[i,j]\rightarrow[i+1,j]$ for $1\leq i<j<n$,
\item $[i]\rightarrow[i,i]\rightarrow[i+1]$ for $1\leq i<n$.
\end{itemize}
(see Example \ref{Ex:An} and Example \ref{Ex:A3} for the case $n=3$.) We have minimal projective resolutions $$0\rightarrow P_{j+1}\rightarrow P_i\rightarrow U_{i,j}\rightarrow 0$$
for all $1\leq i\leq j<n$. Using the fact that all non-zero maps between the indecomposable projectives are scalar multiples of the natural embeddings induced by the chain $P_n\subset\ldots\subset P_1$, we can easily verify that the algebra $B_\mathcal{Q}$ is given as the path algebra of $\widehat{\mathcal{Q}}$ modulo all commutativity relations. The representation $\widehat{M}$ of $\widehat{\mathcal{Q}}$ is given by $M_{[i]}=M_i$ for all $1\leq i\leq n$ and $M_{[i,j]}={\rm Im}(f_j\circ\ldots\circ f_i:M_i\rightarrow M_{j+1})$ for $1\leq i\leq j<n$. The maps representing the arrows of $\widehat{\mathcal{Q}}$ are either natural inclusions or induced by the maps $f_i$.\\[1ex]
To explicitly write down the desingularization map, it is thus necessary to determine the generic subrepresentation types; no general formula is known for these (see however the case $A_2$ below). We restrict to a special case where ${\rm Gr}_{\bf e}(M)$ is known to be irreducible, namely the type $A_n$ degenerate flag variety of \cite{CFR}. We define $M_1=\ldots=M_n=k^{n+1}$, in which we choose a basis $w_1,\ldots,w_{n+1}$, and define $f_i$ as the projection along $w_{i+1}$, that is, $f_i(w_{i+1})=0$ for $i=1,\ldots,n-1$ and $f_i(w_j)=w_j$ for all $i=1,\ldots,n-1$ and $j=1,\ldots,n+1$ such that $j\not=i+1$. Then $M\simeq k\mathcal{Q}\oplus(k\mathcal{Q})^*$. We also define $e_i=i$ for $i=1,\ldots,n$. Then ${\rm Gr}_{\bf e}(M)$ is irreducible with only generic subrepresentation type being $N=k\mathcal{Q}$. It follows immediately from the above description of $\widehat{M}$ that the desingularization coincides with the one defined in \cite{FF}, where this variety is proved to be a isomorphic to a tower of ${\bf P}^1$-bundles.

\subsection{Smooth locus}
In the second example, we show that, in general, our desingularization does not reduce to an isomorphism over the smooth locus, i.e. its fibres can be nontrivial even over smooth points. Namely, consider the quiver $\mathcal{Q}$ given by $1\stackrel{\alpha}{\rightarrow}2$ and the representation $M$ given by $M_1=\langle v_1,v_2,v_3\rangle$, $M_2=\langle w_1,w_2\rangle$, $M_\alpha(v_1)=w_1$, $M_\alpha(v_2)=w_2$, $M_\alpha(v_3)=0$, which is injective, hence exceptional. For ${\bf e}=(1,2)$, the quiver Grassmannian ${\rm Gr}_{\bf e}(M)$ is isomorphic to the projective plane, hence smooth and irreducible. The only generic subrepresentation type $N$ is a generic representation of dimension vector ${\bf e}$. Calculating $\widehat{M}$ and $\widehat{N}$ as above, we see that $\widehat{M}$ is given by $M_1\stackrel{M_\alpha}{\rightarrow}M_2\stackrel{{\rm id}}{\rightarrow}M_2$, and $\dimv\widehat{N}=(1,1,2)$. Now ${\rm Gr}_{\dimv\widehat{N}}(\widehat{M})$ is easily seen to be isomorphic to the blowup of the projective plane in a single point, corresponding to a non-generic subrepresentation of $M$. Note, however, that the desingularization is an isomorphism over the smooth locus in the case of the degenerate flag variety discussed above, as is proved in \cite{CFR2}.

\subsection{$A_2$ case}
Now we give a complete analysis of the $A_2$ case. We start with a general remark on how to approach the description of the ${\rm Aut}(M)$-orbits in ${\rm Gr}_{\bf e}(M)$ in small cases. Consider the quiver $\mathcal{Q}\times A_2$ with vertices $i$ and $i'$ for all $i\in \mathcal{Q}_0$ and with arrows $\alpha:i\rightarrow j$, $\alpha':i'\rightarrow j'$ for all $\alpha:i\rightarrow j$ in $\mathcal{Q}$ and $\iota_i:i'\rightarrow i$ for all $i\in \mathcal{Q}_0$. We consider the algebra $k\mathcal{Q}\otimes kA_2$ which is the quotient of the path algebra $k(\mathcal{Q}\times A_2)$ modulo the ideal generated by all commutativity relations $\alpha\iota_i=\iota_j\alpha'$ for all $\alpha:i\rightarrow j$ in $\mathcal{Q}$. Given $M$ and ${\bf e}$ as before, we consider the dimension vector ${\bf f}$ for $\mathcal{Q}\times A_2$ given by $f_i=d_i$ and $f_{i'}=e_i$. The variety $R_{\bf f}(\mathcal{Q}\times A_2)$ of representations of $\mathcal{Q}\times A_2$ of dimension vector ${\bf f}$ admits a projection map to $R_{\bf d}(\mathcal{Q})$ by restricting to the vertices $i$. Inside $R_{\bf f}(\mathcal{Q}\times A_2)$, we consider the locally closed subset $Y$ consisting of representations of $k\mathcal{Q}\otimes kA_2$ such that all arrows $\iota_i$ are represented by injections. By the definitions, the induced projection $p:Y\rightarrow R_{\bf d}(\mathcal{Q})$ is isomorphic to the universal quiver Grassmannian ${\rm Gr}_{\bf e}^\mathcal{Q}({\bf d})\rightarrow R_{\bf d}(\mathcal{Q})$ of \cite{CFR}, thus $p^{-1}(\mathcal{O}_M)$ is isomorphic to the variety $X_{\bf e}(M)$ of the proof of Proposition \ref{strataclosure}, that is, it is a $G_{\bf e}$-principal bundle over ${\rm Gr}_{\bf e}(M)$. This immediately yields a correspondence between ${\rm Aut}(M)$-orbits in ${\rm Gr}_{\bf e}(M)$ and $G_{\bf f}$-orbits in $p^{-1}(\mathcal{O}_M)$, respecting orbit closure relations, types of singularities, etc.. Furthermore, the latter orbits are in natural bijection to the isomorphism classes of representations $V$ of $k\mathcal{Q}\otimes kA_2$ of dimension vector ${\bf f}$ such that $V$ identifies with $M$ under restriction to the vertices $i$, and such that all $V_{\iota_i}$ are represented by injections.\\[1ex]
This approach to the study of ${\rm Gr}_{\bf e}(M)$ is only efficient once the class of representations $V$ above is well-understood, but the algebra $k\mathcal{Q}\otimes kA_2$ is wild in general.\\[1ex]
Here we only consider the case of the quiver $\mathcal{Q}$ given by $1\stackrel{\alpha}{\rightarrow}2$.
We fix a representation $M$ of $\mathcal{Q}$ of dimension vector ${\bf d}=(d_1,d_2)$, which is thus determined by the rank $r\leq\min(d_1,d_2)$ of the map representing the single arrow, and a dimension vector ${\bf e}=(e_1,e_2)$ such that $e_1\leq d_1$ and $e_2\leq d_2$. The quiver Grassmannian ${\rm Gr}_{\bf e}(M)$ is thus given as the variety of pairs of subspaces $(U_1,U_2)\in{\rm Gr}_{e_1}(M_1)\times{\rm Gr}_{e_2}(M_2)$ such that $M_\alpha(U_1)\subset U_2$.

\begin{prop} The quiver Grassmannian ${\rm Gr}_{\bf e}(M)$ for type $A_2$ has the following geometric properties:
\begin{enumerate}
\item\label{nr1} It is non-empty if and only if $r\leq d_1-e_1+e_2$.
\item\label{nr2} It is reduced and connected.
\item\label{nr3} The ${\rm Aut}(M)$-orbits $\mathcal{O}(r',r'')$ in ${\rm Gr}_{\bf e}(M)$ are uniquely determined by the ranks $r'$ resp.~$r''$ of the induced maps $M_\alpha:U_1\rightarrow U_2$ and $M_\alpha:M_1/U_1\rightarrow M_2/U_2$.
\item\label{nr4} If $r\geq e_1-e_2+d_2$, it is irreducible of dimension $\langle{\bf e},{\bf d}-{\bf e}\rangle$.
\item\label{nr5} If $r<e_1-e_2+d_2$, the irreducible components $I(a)$ of ${\rm Gr}_{\bf e}(M)$ are parameterized by the $a$ such that $\max(0,r+e_1-d_1,r-d_2+e_2)\leq a\leq\min(e_1,e_2,r)$, namely, $I(a)$ consists of all pairs $(U_1,U_2)$ such that the ranks $r'$, $r''$ fulfill $r'\leq a$ and $r''\leq r-a$.
\item\label{nr6} If $e_1=0$ or $e_2=d_2$ or $r=\min(d_1,d_2)$ the variety ${\rm Gr}_{\bf e}(M)$ is smooth.
\item\label{nr7} If $r\geq e_1-e_2+d_2$, the smooth locus consists of all $(U_1,U_2)$ such that $r'=e_1$ or $r''=d_2-e_2$.
\item\label{nr8} If $r<e_1-e_2+d_2$, and the conditions of {\em (\ref{nr6})} are not fulfilled, the smooth locus consists of all $(U_1,U_2)$ such that $r'=a$ and $r''=r-a$ for one of the integers $a$ as in
{\em (\ref{nr5})}.
\end{enumerate}
\end{prop}

\begin{proof} The algebra $kA_2\otimes kA_2$ is of finite representation type; the Auslander-Reiten quiver of the subcategory $\mathcal{D}$ of representations with the arrows $\iota_i$ being represented by injections is of the form
$$\begin{array}{ccccccccc}
&&\begin{array}{cc}1&1\\ 0&0\end{array}&&&&\begin{array}{cc}1&1\\ 1&1\end{array}&&\\
&\nearrow&&\searrow&&\nearrow&&\searrow&\\
\begin{array}{cc}0&1\\ 0&0\end{array}&&&&\begin{array}{cc}1&1\\ 0&1\end{array}&&&&\begin{array}{cc}1&0\\ 1&0\end{array}\\
&\searrow&&\nearrow&&\searrow&&\nearrow&\\
&&\begin{array}{cc}0&1\\ 0&1\end{array}&&&&\begin{array}{cc}1&0\\ 0&0\end{array}&&\end{array}$$

A representation in $\mathcal{D}$ is thus completely determined up to isomorphism by the system
$$\begin{array}{ccccc}
&p_2&&p_5&\\ p_1&&p_4&&p_7\\ &p_3&&p_6& \end{array}$$
of multiplicities of the above indecomposables. A short calculation shows that the representations $V(r',r'')$ in $\mathcal{D}$ of dimension vector $\begin{array}{cc}d_1&d_2\\ e_1&e_2\end{array}$ restricting to $M$ are given by the multiplicities
$$\begin{array}{ccl}
p_1&=&d_2-e_2-r''\\
p_2&=&r''\\
p_3&=&e_2-r+r''\\
p_4&=&r-r'-r''\\
p_5&=&r'\\
p_6&=&d_1-e_1-r+r'\\
p_7&=&e_1-r'',\end{array}$$
in terms of parameters $r',r''$, which thus have to fulfill the inequalities
$$\begin{array}{ccl}
\max(0,r+e_1-d_1)&\leq &r'\leq e_1,\\
\max(0,r-e_2)&\leq &r''\leq d_2-e_2,\\
r'+r''&\leq r;\end{array}$$
we denote by $R$ the subset of ${\bf N}^2$ of pairs $(r',r'')$ fulfilling these inequalities.\\[1ex]
Thus the ${\rm Aut}(M)$-orbits $\mathcal{O}(r',r'')$ in ${\rm Gr}_{\bf e}(M)$ are naturally indexed by these parameters. Moreover, the parameters $r',r''$ are chosen in such a way that a subrepresentation $U\in\mathcal{O}(r',r'')$ is a representation of dimension vector ${\bf e}$, with the map representing the unique arrow of $\mathcal{Q}$ being of rank $r'$, and the corresponding factor representation $M/U$ is of dimension vector ${\bf d}-{\bf e}$, with the map representing the unique arrow of $\mathcal{Q}$ being of rank $r''$. This proves claim (\ref{nr3}). Moreover, working out the condition for non-emptyness of $R$, we arrive at claim (\ref{nr1}).\\[1ex]
We can also work out the orbit closure relation using the description of degenerations of representations of $k\mathcal{Q}\otimes kA_2$ (which is a representation directed algebra) in terms of the so-called ${\rm Hom}$-ordering \cite{Bo}. A straightforward calculation yields the following criterion:

\begin{center} We have $\mathcal{O}(r_1',r_1'')\subset\overline{\mathcal{O}(r_2',r_2'')}$ if and only if $r_1'\leq r_2'$ and $r_1''\leq r_2''$.
\end{center}

With the aid of this criterion, we can determine the irreducible components of ${\rm Gr}_{\bf e}(M)$ as the closures of the maximal (with respect to orbit closure inclusion) orbits, yielding claim (\ref{nr5}) and the first half of claim (\ref{nr4}). Since any two different irreducible components $I(a)$, $I(a')$ intersect, namely in the closure of the orbit $\mathcal{O}(a,r-a')$, we have proved the second half of claim (\ref{nr2}).

By computing the dimension of the endomorphism ring of the representation $V(r',r'')$, we can determine the dimension of the orbit $\mathcal{O}(r',r'')$ as
$$e_1(d_1-e_1)+e_2(d_2-e_2)-(d_2-e_2+e_1)r+(e_1+r)r'+(d_2-e_2+r)r''-r'^2-r'r''-r''^2.$$
This yields the second half of claim (\ref{nr4}).

The dimension of the tangent space to a point $U\in\mathcal{O}(r',r'')$ can be computed, using the formula $\dim T_U({\rm Gr}_{\bf e}(M))=\dim{\rm Hom}(U,M/U)$, as
$$e_1(d_1-e_1)+e_2(d_2-e_2)-(d_2-e_2)r'-e_1r''+r'r''.$$
This yields claim (\ref{nr7}), as well as claim (\ref{nr8}) using that all all non-maximal orbits belong to the intersection of at least two irreducible components in this case. Finally, the first half of claim (\ref{nr2}) follows.
\end{proof}
Specializing the general properties of the desingularization in the present case, we arrive at:

\begin{cor} The following properties of the desingularization of ${\rm Gr}_{\bf e}(M)$ hold:
\begin{enumerate}
\item\label{no1} If $r\geq e_1-e_2+d_2$, the fibre of the desingularization
over a point of $\mathcal{O}(r',r'')$ is isomorphic to the Grassmannian ${\rm Gr}_{e_1-r'}(k^{r-r'-r''})$.
\item\label{no2} If $r\geq e_1-e_2+d_2$, the desingularization is one to one over the smooth locus if and only if $r=e_1+d_2-e_2$.
\item\label{no3} In this case, it is even small.
\item\label{no4} If $r<e_1-e_2+d_2$, the fibre of the desingularization
over a point of $\mathcal{O}(r',r'')$ is isomorphic to the disjoint union of the Grassmannians ${\rm Gr}_{a-r'}(k^{r-r'-r''})$ for $\max(0,r+e_1-d_1,r-d_2+e_2,r')\leq a\leq\min(e_1,e_2,r,r-r'')$.
\item\label{no5} In this case, unless $e_1=0$ or $e_2=d_2$ or $r=\min(d_1,d_2)$, the desingularization is one to one over the smooth locus.
\end{enumerate}
\end{cor}

Concluding the discussion of the $A_2$ case, we remark that the case $r=e_1+e_2-d_2$ is precisely the case of quiver Grassmannians of the form ${\rm Gr}_{\dimv P}(P\oplus I)$ for $P$ a projective and $I$ and injective representation studied in \cite{CFR}. An open question is whether the desingularization is one to one over the smooth locus in this case for arbitrary Dynkin quivers.

\subsection{Del Pezzo surface}
Now we consider the quiver $\mathcal{Q}$ given by $1\rightarrow 2\leftarrow 3$ and the quiver Grassmannian $X={\rm Gr}_{\dimv k\mathcal{Q}}(k\mathcal{Q}\oplus k\mathcal{Q}^*)$, which is thus a generalized degenerate flag variety in the sense of \cite{CFR}. Choosing appropriate basis, the representation $k\mathcal{Q}\oplus k\mathcal{Q}^*$ can be written as
$$
\xymatrix@C=50pt{
k^3\ar[r]^{\left(\tiny{\begin{array}{ccc}1&0&0\\ 0&1&0\\ 0&0&0\\ 0&0&0\end{array}}\right)}&k^4&k^3
\ar[l]_{\left(\tiny{\begin{array}{ccc}0&0&0\\ 1&0&0\\ 0&1&0\\ 0&0&0\end{array}}\right)}
}
$$
The dimension vector ${\dimv k\mathcal{Q}}$ equals $(1,3,1)$, thus, identifying ${\rm Gr}_3(k^4)$ with ${\bf P}^3$,
the quiver Grassmannian $X$ can be realized as
$$
\{((a:b:c),(d:e:f),(n:p:q:r))\in{\bf P}^2\times{\bf P}^2\times{\bf P}^3\, :\, an+bp=0,\, dp+eq=0\},$$
which is a singular projective variety of dimension five. We work out the desingularization $Y$ in this specific case. The quiver $\widehat{\mathcal{Q}}$ is of type $E_6$ and it is shown in Example \ref{Ex:A3noequi}. The representation $\widehat{k\mathcal{Q}\oplus k\mathcal{Q}^*}$ of $\widehat{\mathcal{Q}}$ admits the following explicit form:
$$
\xymatrix@C=60pt@R=60pt{
&&k^4&&\\
k^3\ar[r]^{\left(\tiny{\begin{array}{ccc}1&0&0\\ 0&1&0\end{array}}\right)}&k^2\ar[r]^{\left(\tiny{\begin{array}{cc}1&0\\ 0&1\\ 0&0\end{array}}\right)}&k^3\ar[u]^(.6){\left(\tiny{\begin{array}{ccc}1&0&0\\ 0&1&0\\ 0&0&1\\ 0&0&0\end{array}}\right)}&k^2\ar[l]_{\left(\tiny{\begin{array}{cc}0&0\\ 1&0\\ 0&1\end{array}}\right)}&k^3\ar[l]_{\left(\tiny{\begin{array}{ccc}1&0&0\\ 0&1&0\end{array}}\right)}
}
$$
The only generic subrepresentation type being $N=k\mathcal{Q}$, we thus have to consider subrepresentations of dimension vector
$$\xymatrix@R=1pt@C=1pt{&&3&&\\1&1&2&1&1}
$$
of this representation. Again identifying ${\rm Gr}_2(k^3)$ with ${\bf P}^2$, we arrive at the following realization of $Y$:
$$\{((a:b:c),(d:e:f),(g:h),(i:j),(k:l:m),(n:p:q:r))$$
$$\in{\bf P}^2\times{\bf P}^2\times{\bf P}^1\times{\bf P}^1\times{\bf P}^2\times{\bf P}^3\, :$$
$$ah=bg,\, dj=ei,\, kp=ln,\, kq=mn,\, lq=mp,\, gk+hl=0,\, il+jm=0\},$$
with the desingularization map being the projection to the first, second and sixth component. Defining $Z$ as
$$\{((g:h),(i:j),(k:l:m))\in{\bf P}^1\times{\bf P}^1\times{\bf P}^2\, :\,  gk+hl=0,\, il+jm=0\},$$
we can view $Y$ as a closed subvariety of $X\times Z$, with the desingularization map being the first projection.\\[1ex]
The structure of $Z$ is easily analysed by considering the projection to ${\bf P}^2$; namely, this proves that $Z$ is isomorphic to a Del Pezzo surface, namely ${\bf P}^2$ blown up in two distinct points. By a straightforward analysis of the projection from $Y$ to $Z$, we can see that $Y$ is a three-fold tower of ${\bf P}^1$-fibrations over $Z$. Thus, the Poincar\'e polynomial of $Y$
(in $l$-adic cohomology  for an arbitrary algebraically closed field $k$)
equals $(1+3t^2+t^4)(1+t^2)^3$.\\[1ex]
The only two-dimensional fibre of the desingularization map is the one over the point $((0:0:1),(0:0:1),(0:0:0:1))$, namely, it is isomorphic to the Del Pezzo surface $Z$. If $(a,b)\not=0$ or $(d,e)\not=0$ or $(n,p,q)\not=0$, the fibre is trivial, thus the locus of points of $X$ with positive dimensional fibre is of codimension at least three (compare this to the general result of \cite{CFR2} that $X$ is regular in codimension two), proving smallness of the desingularization map. Consequently, we also know the Poincar\'e polynomial of the ($l$-adic) intersection cohomology of $X$.

\section*{Acknowlegments}
The authors would like to thank Claire Amiot, Klaus Bongartz, Aslak Buan, Lutz Hille, Bernhard Keller, Ragnar-Olaf Buchweitz and Michel Van den Bergh for interesting remarks on this work, and Sarah Scherotzke for pointing out a gap in an earlier version.\\[1ex]
The work of Evgeny Feigin was partially supported
by the Russian President Grant MK-3312.2012.1, by the Dynasty Foundation,
by the AG Laboratory HSE, RF government grant, ag. 11.G34.31.0023, by the RFBR grants
12-01-33101, 12-01-00070, 12-01-00944 and by the Russian Ministry of Education 
and Science under the grant 2012-1.1-12-000-1011-016.
This study comprises research fundings from the `Representation Theory
in Geometry and in Mathematical Physics' carried out within The
National Research University Higher School of Economics' Academic Fund Program
in 2012, grant No 12-05-0014.
This study was carried out within the National Research University Higher School of Economics
Academic Fund Program in 2012-2013, research grant No. 11-01-0017.

The work of Giovanni Cerulli Irelli was supported by SFB Transregio 45, "Periods, moduli spaces and arithmetic of algebraic varieties"
Bonn - Mainz - Essen.

\bibliographystyle{amsplain}

\end{document}